\documentclass[a4,reqno]{amsart}
\usepackage[english]{babel}
\usepackage{xcolor}
\usepackage{enumerate}

\usepackage[colorlinks=false,backref=false,pagebackref=false,pdftex,pdfauthor={D. Bonheure, E. Moreira dos Santos, H. Tavares},pdftitle={}]{hyperref}
\usepackage{amsmath,amsthm,amsfonts,latexsym,amssymb, hyperref}
\usepackage{comment}

\usepackage[textwidth=20mm]{todonotes}

\newcommand{\innt}{{\rm int\,}}

\newcommand{\R}{{\mathbb R}}

\newcommand{\N}{{\mathbb N}}

\newcommand{\cH}{{\mathcal H}}

\newcommand{\cO}{{\mathcal O}}

\usepackage{setspace}
\newcommand{\eps}{\varepsilon}
\newcommand{\vphi}{\varphi}

\def\sideremark#1{\ifvmode\leavevmode\fi\vadjust{\vbox to0pt{\vss
 \hbox to 0pt{\hskip\hsize\hskip1em
 \vbox{\hsize2.1cm\tiny\raggedright\pretolerance10000
  \noindent #1\hfill}\hss}\vbox to15pt{\vfil}\vss}}}%

\numberwithin{equation}{section}
\newtheorem{theorem}{Theorem}[section]
\newtheorem{lemma}[theorem]{Lemma}

\newtheorem{remark}[theorem]{Remark}
\newtheorem{proposition}[theorem]{Proposition}

\newtheorem{corollary}[theorem]{Corollary}

\begin{document}

\title[Nodal Solutions]
{Nodal Solutions for sublinear-type problems with Dirichlet boundary conditions}

%
\author[D. Bonheure]{Denis Bonheure}
\address{Denis Bonheure \newline \indent D{\'e}partement de Math{\'e}matique \newline \indent  Universit{\'e} libre de Bruxelles \newline \indent CP 214,  Boulevard du Triomphe, B-1050 Bruxelles, Belgium}
\email{denis.bonheure@ulb.ac.be}
\author[E. Moreira dos Santos]{Ederson Moreira dos Santos}
\address{Ederson Moreira dos Santos \newline \indent Instituto de Ci{\^e}ncias Matem\'aticas e de Computa\c{c}\~ao - ICMC \newline \indent Universidade de S\~ao Paulo - USP \newline \indent
Caixa Postal 668, CEP 13560-970 - S\~ao Carlos - SP - Brazil}
\email{ederson@icmc.usp.br}
\author[E. Parini]{Enea Parini}
\address{Enea Parini \newline \indent Aix-Marseille Universit\'e, CNRS, Centrale Marseille, I2M, UMR 7373 \newline \indent 39 Rue Frederic Joliot Curie, 13453 Marseille, France}
\email{enea.parini@univ-amu.fr}

\author[H. Tavares]{Hugo Tavares}   
\address{Hugo Tavares \newline \indent CAMGSD and Mathematics Department,
\newline \indent Instituto Superior T\'ecnico, Universidade de Lisboa   \newline \indent
Av. Rovisco Pais, 1049-001 Lisboa, Portugal}
\email{hugo.n.tavares@tecnico.ulisboa.pt}

\author[T. Weth]{Tobias Weth}
\address{Tobias Weth \newline \indent Goethe-Universit\"at Frankfurt, Institut f\"ur Mathematik\newline \indent Robert-Mayer-Str. 10, D-60629 Frankfurt, Germany}
\email{weth@math.uni-frankurt.de}

\date{\today}



\begin{abstract} 
We consider nonlinear second order elliptic problems of the type
\[
-\Delta u=f(u) \text{ in } \Omega, \qquad u=0  \text{ on } \partial \Omega,
\]
where $\Omega$ is an open $C^{1,1}$--domain in $\R^N$, $N\geq 2$,  under some general assumptions on the nonlinearity  that include the case of a sublinear pure power  $f(s)=|s|^{p-1}s$ with $0<p<1$ and of Allen-Cahn type $f(s)=\lambda(s-|s|^{p-1}s)$ with $p>1$ and $\lambda>\lambda_2(\Omega)$ (the second Dirichlet eigenvalue of the Laplacian). We prove the existence of a least energy nodal (i.e. sign changing) solution, and of a nodal solution of mountain-pass type. We then give explicit examples of domains where the associated levels \emph{do not} coincide. For the case where $\Omega$ is a ball or annulus and $f$ is of class $C^1$, we prove instead that the levels coincide, and that least energy nodal solutions are nonradial but axially symmetric functions. Finally, we provide stronger results for the Allen-Cahn type nonlinearities in case $\Omega$ is either a ball or a square. In particular we give a complete description of the solution set for $\lambda\sim \lambda_2(\Omega)$, computing the Morse index of the solutions.
\end{abstract}



\maketitle

\section{Introduction} 
We consider nonlinear second order elliptic equations of the type
\begin{equation}\label{eq:mainproblem}
\begin{cases}
-\Delta u=f(u) &\text{ in } \Omega,\\
u=0 & \text{ on } \partial \Omega,
\end{cases}
\end{equation}
where $\Omega$ is an open $C^{1,1}$--bounded domain in $\R^N$, $N\geq 2$. We assume  general hypotheses on $f$ such as continuity, oddness, strict monotonicity of $f(s)/s$ in $(0,\infty)$ together with sublinear-type conditions. Prototypes of those functions are $f(s)=|s|^{p-1}s$ with $0<p<1$ (sublinear power), or $f(s)=\lambda(s-|s|^{p-1}s)$ with $p>1$ and $\lambda>\lambda_1(\Omega)$ (Allen-Cahn type, also referred to as bistable nonlinearities). Here $(\lambda_k(\Omega))$, for short $(\lambda_k)$, denotes the sequence of eigenvalues of $(\Delta, H^1_0(\Omega))$. The main goal of this paper is to study nodal solutions of \eqref{eq:mainproblem}, with a special emphasis on those having least energy with respect to the associated Euler-Lagrange functional. We provide a unified proof of known existence results and include new results in several directions.

Under our hypotheses, we show that \eqref{eq:mainproblem} has a unique bounded positive solution $w$, and $w$ and $-w$ are the global minimizers of the corresponding energy functional $I: H^1_0(\Omega) \to \mathbb{R}$,
\begin{equation*}
I(u)=\begin{cases}
\frac{1}{2}\int_\Omega |\nabla u|^2-\int_\Omega F(u) & \text{ if } F(u)\in L^1(\Omega),\\
+\infty & \text{ if } F(u)\not \in L^1(\Omega),
\end{cases}
\end{equation*}
with $F(t):=\int_0^t f(s)\, ds$; see Theorem \ref{th:existencemin} and Proposition \ref{prop:uniqueness} ahead. We do not assume any growth condition on $|f|$ at infinity, and for that reason the energy  functional is not always finite. We observe that the existence and uniqueness of these signed solutions are known, e.g. \cite{Berestycki1981, Berger1979}, via sub-super solutions method. Here we give a variational characterization of these solution, which is essential for the construction of a special nodal solution.

Using paths connecting these global minimizers, we prove the existence of a bounded nodal mountain pass solution of \eqref{eq:mainproblem} (Theorem \ref{thm:MP}), at the min-max critical level
\[
c_{mp}=\inf_{\gamma\in \Gamma} \sup_{t\in [0,1]} I(\gamma(t)),\quad \text{ with } \quad \Gamma=\{\gamma\in C([0,1],H^1_0(\Omega)):\ \gamma(0)=-w,\ \gamma(1)=w \}.
\]
Moreover, among the set of bounded nodal solutions, we prove the existence of one that achieves the least energy nodal level (Theorem \ref{thm:lensachieved}), defined as:
\[
c_{nod}=\inf\{I(u):\ u\in H^1_0(\Omega)\cap L^\infty(\Omega), u^\pm\not\equiv 0,\ -\Delta u=f(u)\}.
\]
A natural question is whether $c_{nod}=c_{mp}$. In Theorem \ref{thm:cmp_neqc_nod} we provide an example of a domain, a dumbbell with a thin channel, where this equality does not hold. This shows that, in general, the set of low energy nodal solutions to sublinear problem has a more complicated structure than in the case of linear or superlinear problems. In particular, we point out that in the linear case of the Dirichlet eigenvalue problem for the Laplacian, the second Dirichlet eigenvalue 
corresponds to the least energy of sign changing eigenfunctions, and this eigenvalue has the mountain pass characterization
\begin{equation*}
\lambda_2(\Omega)  	=\inf_{c\in \Lambda} \sup_{t\in [0,1]} \int_\Omega |\nabla c (t)|^2 = \inf\left\{\int_\Omega |\nabla u|^2: u\in H^1_0(\Omega),\ \|u\|_{L^2}=1,\ \text{$u$ is a nodal eigenfunction}\right\},
\end{equation*}
where
\[
\Lambda=\left\{c\in C([0,1],H^1_0(\Omega)),\ c(0)=-\varphi_1,\ c(1)=\varphi_1,\ \|c(t)\|_{L^2}=1\ \forall t\right\},
\]
with $\varphi_1$ being the first $L^2$-normalized positive eigenfunction (see e.g. \cite{CuestaFigueiredoGossez99}).

In the case where $f$ is of class $C^1$, we also use the Morse index to characterize nodal solutions. In Theorem \ref{prop:boundmorseindex} ahead, we show that every bounded solution $u$ of (\ref{eq:mainproblem}) with $I(u)=c_{nod}$ has Morse index $m(u)$ less than or equal to $1$. Here we recall that $m(u)$ is defined as the number of negative Dirichlet eigenvalues of the operator $-\Delta - f'(u)$ in $\Omega$ (counted with multiplicity). 
Moreover, if a bounded solution $u$ with $I(u)=c_{nod}$ satisfies $m(u)=1$, then $u$ is of mountain pass type, and we have $c_{nod}=c_{mp}$ in this case, see Theorem~\ref{sec:lens-mp}. In particular, we shall see that this is the case in bounded radial domains $\Omega$. In this case, we also deduce that every least energy nodal solution $u$ is nonradial and foliated Schwarz symmetric. More precisely, $u$ is axially symmetric and strictly decreasing with respect to the polar angle from the symmetry axis, see Theorem~\ref{sec:ball-case}.

The results mentioned in the last two paragraphs show that least energy nodal solutions of sublinear-type problems with Dirichlet boundary conditions might have different variational characterizations and different Morse indices, zero or one, on different domains. Observe that this is in sharp contrast with the superlinear case, where the Morse index of these solutions is always two; see \cite{BartschChangWang, BartschWeth, CastroCossioNeuberger}. In this paper we provide examples of sets where $c_{nod}=c_{mp}$, and sets where this does not happen. To understand the general picture is an open problem. We conjecture that if $\Omega$ is a convex $C^{1,1}$ domain, then both levels coincide.

In the last part of this work, we focus on the specific Allen-Cahn type problem
\begin{equation}\label{eq:AC-intro}
\begin{cases}
-\Delta u=\lambda(u-|u|^{p-1}u)\\
 u\in H^1_0(\Omega)
 \end{cases}\qquad (p>1,\ \lambda>\lambda_2(\Omega)).
\end{equation}
proving complementary results for regular domains and in the case of a simple Lipschitz domain which is nonregular: a square in dimension 2. Given any $C^{1,1}$--domain $\Omega$, there exists $\eps>0$ such that $c_{nod}=c_{mp}$ whenever $\lambda\in (\lambda_2(\Omega),\lambda_2(\Omega)+\eps)$, see Theorem \ref{thm:Morselambdasmall}. We describe and characterize completely the set of bounded nodal solutions in the case of the ball (Theorem \ref{thm:symmetry}) and in the case of a square in dimension 2 (Theorem \ref{th:delpino}), still for $\lambda\sim \lambda_2(\Omega)$; the set is the union of branches that start at the eigenspace associated with $\lambda_2(\Omega)$. In the case of the ball all solutions are least energy nodal solutions; in the square there are exactly four branches of solutions, and we compute their Morse indices and energies,  showing which branches correspond to solutions with minimal energy (Theorem \ref{th:square}). The results in this paragraph complement previous results by Miyamoto \cite{Miyamoto} (in the ball) and del Pino, Garc\'ia-Meli\'an, Musso \cite{delPinoGarciaMelianMusso} (in the square).

We conclude this introduction with some additional literature on related problems. Via min-max methods, it is shown in  \cite{BartschWillem} that sublinear-type problems like \eqref{eq:mainproblem} have actually infinitely many nodal solutions. The pure power case $f(u)=|u|^{p-1}u$ ($1<p<2$) with \emph{Neumann} boundary conditions is treated in \cite{PariniWeth} (see also \cite[Corollary 1.4]{SaldanaTavares}). In this case all nontrivial solutions change sign and there exists a least energy nodal solution, with variational characterization
\[
c_{nod}=\inf\left\{\frac{1}{2}\int_\Omega |\nabla u|^2-\frac{1}{p+1}\int_\Omega |u|^{p+1}:\ u\in H^1(\Omega),\ \int_\Omega |u|^{p-1}u=0  \right\}.
\]
Whenever the domain is a ball or an annulus, minimizers are not radial but merely foliated Schwarz symmetric. These results were extended to Lane-Emden system in \cite{SaldanaTavares}. The existence of infinitely many solutions for the single equation with Neumann boundary condition is shown in \cite{Du}.

\section{Positive Solutions and properties of bounded solutions}\label{section:positivesolutions}

Let $\Omega$ be either an open $C^{1,1}$--domain in $\R^N$, $N\geq 2$,  and let $f:\R\to \R$ be a function satisfying the following assumptions:
\begin{itemize}
\item[(A1)] $f$ is odd and continuous;
\item[(A2)] $s\mapsto \frac{f(s)}{s}$ is (strictly) decreasing in $(0,\infty)$;
\item[(A3)] $\lim \limits_{s\to 0} \frac{f(s)}{s}>\lambda_1(\Omega)$;
\item[(A4)] $\lim \limits_{s\to \infty} \frac{f(s)}{s}<\lambda_1(\Omega)$.
\end{itemize}

Consider $F(t):=\int_0^t f(s)\, ds$. Associated to \eqref{eq:mainproblem}, set the energy functional $I:H^1_0(\Omega)\to \overline\R$ defined by
\begin{equation}\label{eq:functional_I}
I(u)=\begin{cases}
\frac{1}{2}\int_\Omega |\nabla u|^2-\int_\Omega F(u) & \text{ if } F(u)\in L^1(\Omega),\\
+\infty & \text{ if } F(u)\not \in L^1(\Omega),
\end{cases}
\end{equation}
and the least energy level
\[
m=m(\Omega):=\inf\{I(v):\ v\in H^1_0(\Omega)\}.
\]

\begin{lemma}[$m$ is achieved]\label{lemma:m_achieved} We have $-\infty<m<0$, and there exists $u\in H^1_0(\Omega)\backslash\{0\}$ such that $I(u)=m$.
\end{lemma}

\begin{proof}
1) Let us prove that $m<0$. Let $\varphi_1$ be the positive, $L^2$--normalized eigenfunction associated to $\lambda_1(\Omega)$. From assumption (A3) we deduce the existence of $\bar s,\delta>0$ such that 
\[
F(s)\geq \frac{(1+\delta)\lambda_1(\Omega)}{2}s^2\qquad \text{ for every $0\leq s\leq \bar s$}.
\]
By choosing $\varepsilon>0$ small so that $\varepsilon \|\varphi_1\|_\infty<\bar s$, 
\begin{align}
I(\varepsilon \varphi_1)=\frac{\varepsilon^2}{2}\int_\Omega |\nabla \varphi_1|^2-\int_\Omega F(\varepsilon \varphi_1) \leq \frac{\varepsilon^2}{2}\lambda_1(\Omega) - \varepsilon^2 \frac{1+\delta}{2}\lambda_1(\Omega)=-\frac{\varepsilon^2}{2}\delta \lambda_1(\Omega)<0.
\end{align}
\smallbreak

\noindent 2) We now prove that $I$ is coercive. Assumption (A4) implies the existence of $\varepsilon,C>0$ such that $F(s)\leq C+ \frac{(1-\varepsilon) \lambda_1(\Omega)}{2}s^2$ for all $s\in\R$. Thus, given $u\in H^1_0(\Omega)$ such that $F(u)\in L^1(\Omega)$, we have
\[
I(u)\geq \frac{1}{2}\int_\Omega |\nabla u|^2-C|\Omega| -\frac{(1-\varepsilon) }{2}\lambda_1(\Omega)\int_\Omega u^2 \geq \frac{\varepsilon}{2}\int_\Omega |\nabla u|^2-C|\Omega|.
\]
This implies that $m>-\infty$, as well as the coercivity of $I$.
\smallbreak
\noindent 3) Let us check that $m$ is achieved. Take a minimizing sequence $u_n$, which by the coercivity of $I$ is uniformly bounded in $H^1_0(\Omega)$. Then, up to a subsequence, $u_n \rightharpoonup u$ weakly in $H^1_0(\Omega)$, $u_n\to u$ strongly in $L^2(\Omega)$, and there exists $h\in L^2(\Omega)$ such that $|u_n|\leq h$ a.e. in $\Omega$, for every $n$ . We have
\[
F(u_n)\leq C+\frac{\lambda_1(\Omega)u_n^2}{2}\leq C+\frac{\lambda_1(\Omega)h^2}{2}
\]
and by the reverse Fatou lemma and the fact that $F(u)\leq C+\frac{\lambda_1(\Omega)h^2}{2}$ we have
\[
\limsup_{n \to +\infty} \int_\Omega F(u_n) \leq \int_\Omega \limsup_{n \to +\infty} F(u_n)=\int_\Omega F(u)<+\infty.
\]
Thus
\[
m=\lim_{n \to +\infty} I(u_n) \geq \liminf_{n \to +\infty} \frac{1}{2}\int_\Omega |\nabla u_n|^2-\limsup_{n \to +\infty} \int_\Omega F(u_n) \geq \frac{1}{2}\int_\Omega |\nabla u|^2-\int_\Omega F(u).
\]
So $F(u)\in L^1(\Omega)$, and $u$ achieves $m$.
\end{proof}

\begin{remark}\label{rem_energynegative}
For future reference, we remark that in paragraph 1) of the previous proof we showed that $I(\varepsilon \varphi_1)<0$ for small $\varepsilon>0$.
\end{remark}
Our aim now is to show the following.

\begin{theorem}\label{th:existencemin}
Assume that $f$ satisfies (A1)-(A4).
Let $u\in H^1_0(\Omega)$ achieve $m$. Then $u\in L^\infty(\Omega)$, $|u|>0$ in $\Omega$, and $u$ is a weak solution of \eqref{eq:mainproblem}.
\end{theorem}
\begin{proof}
Let $u\in H^1_0(\Omega)$ be a minimizer of $I$, and along this proof let us denote $v:=|u|\in H^1_0(\Omega)$. Since $f$ is odd, then $F$ is even and $I(u)=I(v)$. In this proof we denote by $f^+(t):=\max\{f(t),0\}$ and $f^-(t):=\max \{-f(t),0\}$ respectively the positive and negative parts of $f$.
\smallbreak

\noindent 1) Let us check that $f^+(v)\in  L^1(\Omega)$ and that 
\begin{equation}\label{eq:subsolution}
-\Delta v\leq f^+(v) \qquad \text{weakly in $\Omega$}.
\end{equation}  
First of all, observe that by the continuity of $f$ and property (A4), we have
\begin{equation}\label{eq:bound_f^+}
0\leq f^+(s)\leq C+|s|\lambda_1(\Omega).
\end{equation}
Since $v\in H^1_0(\Omega)$, this yields $f^+(v)\in L^1(\Omega)$. As for the second claim, take a test function $\varphi\in C^\infty_c(\Omega)$, $\varphi\geq 0$. Then, by the minimality property of $v$,
\begin{align}
0\leq \liminf_{t\to 0^+} \frac{I(v-t\varphi)-I(v)}{t}&=-\int_\Omega \nabla v\cdot \nabla \varphi- \limsup_{t\to 0^+} \frac{1}{t} \int_\Omega (F(v-t\varphi)-F(v)) \\
   & = -\int_\Omega \nabla v\cdot \nabla \varphi + \liminf_{t\to 0^+} \int_\Omega \int_0^1 f(v-t\tau \varphi)\varphi\, d\tau\, dx\\
   &\leq -\int_\Omega \nabla v\cdot \nabla \varphi + \liminf_{t\to 0^+} \int_\Omega \int_0^1 f^+(v-t\tau \varphi)\varphi\, d\tau\, dx\\
   &=-\int_\Omega \nabla v\cdot \nabla \varphi + \int_\Omega  f^+(v)\varphi\, dx,
\end{align}
where to get the second to last inequality we used the sign of $\varphi$, whereas to get the convergence on the last equality we used Lebesgue's dominated convergence theorem and \eqref{eq:bound_f^+}.
\smallbreak

\noindent 2) We check that $v\in L^\infty(\Omega)$. This is a consequence of rewriting  the inequality \eqref{eq:subsolution} deduced in paragraph 1 as
\[
-\Delta v -V(x) v\leq g
\]
with
\[
V(x)=\frac{f^+(v)}{v}\chi_{\{v>1\}}\in L^\infty(\Omega),\qquad g=f^+(v) \chi_{\{0\leq v\leq 1\}} \in L^\infty(\Omega)
\]
($V\in L^\infty(\Omega)$ by (A2)). Then \cite[Theorem 8.15]{GilbargTrudinger} yields $v\in L^\infty(\Omega)$.
\smallbreak

\noindent 3) Since $v\in L^\infty(\Omega)$, then also $u\in L^\infty(\Omega)$ and we can compute, for any $\varphi\in C^\infty_c(\Omega)$,
\[
0 =\frac {d}{dt} I(u+t\varphi)|_{t=0}= \int_\Omega \nabla u\cdot \nabla \varphi- \int_\Omega f(u) \varphi,
\]
hence $u$ is a weak solution of \eqref{eq:mainproblem}. 

\smallbreak

\noindent 4) Finally, let us check that $v=|u|>0$. Since $F$ is even, then $v$ also achieves $m$ and by the previous paragraph it solves  $-\Delta v=f(v)$. By elliptic regularity  and Sobolev embeddings we have $v\in W^{2,p}\cap C^{1,\alpha}(\overline \Omega)$ for every $p\geq 1$, $\alpha\in (0,1)$. By (A3), there exists $\delta>0$ such that $f(s)>0$ for $0< s\leq \delta$. Thus we can apply the strong maximum principle (\cite[Theorem 9.6]{GilbargTrudinger}) to
\[
-\Delta v-c(x)v=f^+(v) \chi_{\{v>\delta\}}+f(v) \chi_{\{0\leq v\leq \delta\}}\geq 0
\]
with
\[
c(x)=\frac{f^-(v)}{-v}\chi_{\{v>\delta\}}\leq 0
\]
to obtain that either $v\equiv 0$ or $v>0$. Since $m<0$, the former case cannot occur.
\end{proof}

\begin{proposition} \label{prop:uniqueness} Assume that $f$ satisfies (A1)-(A4). Then there exists a unique positive bounded solution of \eqref{eq:mainproblem}.
\end{proposition}
\begin{proof} Since we are assuming (A2), this is a direct consequence of \cite{BrezisOswald}, see also of \cite[Appendix II]{BrezisKamin}.
\end{proof}

\noindent \textbf{Notation.} From now on, we will denote by $w$ the unique positive bounded solution of \eqref{eq:mainproblem}.

\begin{lemma}\label{lemma:uniformbounded} Let $f$ be a function satisfying (A1)-(A2). Let $u$ be any bounded solution of \eqref{eq:mainproblem}. Then 
\[
-w(x)\leq u(x)\leq w(x) \text{ for every } x\in \Omega.
\]
\end{lemma}
\begin{proof}
Let $\Omega':=\{x\in \Omega:\ u(x)>w(x)\}$ and assume by contradiction that $|\Omega'|>0$. Take a sequence $\varepsilon_n\to 0^+$  of regular values of $u-w$. Then $\Omega_n':=\{x\in \Omega:\ u(x)-w(x)>\varepsilon_n\}\Subset \Omega$ is a $C^{1,\alpha}$--domain and we can integrate by parts on it:
\begin{align}
\int_{\Omega_n'} \Bigl(\frac{f(w)}{w}-\frac{f(u)}{u}\Big)wu &=\int_{\Omega_n'} ((-\Delta w) u- (-\Delta u)w) = \int_{\partial \Omega_n'} \Bigl(\frac{\partial u}{\partial \nu} w - \frac{\partial w}{\partial \nu} u\Bigr)\\
&=\int_{\partial \Omega_n'} w\Bigl(\frac{\partial u}{\partial \nu}-\frac{\partial w}{\partial \nu}\Bigr)- \varepsilon_n \int_{\partial \Omega'_n} \frac{\partial w}{\partial \nu}
\end{align}
Since $w>0$ in $\Omega$ and $\frac{\partial }{\partial \nu} (u-w)\leq 0$ on $\partial \Omega_n'$, then  
\[
\int_{\partial \Omega_n'} w\Bigl(\frac{\partial u}{\partial \nu}-\frac{\partial w}{\partial \nu}\Bigr)\leq 0.
\]
As for the last term in the inequality,
\[
\varepsilon_n \int_{\partial \Omega'_n} \frac{\partial w}{\partial \nu}= \varepsilon_n \int_{\Omega_n'} \Delta w = -\varepsilon_n \int_{\Omega_n'} f(w)\to 0.
\]
Thus we conclude that 
\[
\int_{\Omega'}  \Bigl(\frac{f(w)}{w}-\frac{f(u)}{u}\Big)uw\leq 0,
\]
a contradiction by (A2) and $|\Omega'|>0$. Therefore $u(x)\leq w(x)$ for a.e. $x\in \Omega$. Analogously, one shows that $u(x)\geq -w(x)$.
\end{proof}

\section{Properties and Truncation of the functional $I$} \label{subsec:truncation}

Throughout this section we assume that $f$ satisfies (A1)-(A4).  By (A1) and (A2), there exists at most one zero of $f$ on the half-line $]0,+\infty[$. Therefore we might have two situations: either

\smallbreak
\begin{itemize}
\item[($f1$)] we have $f(s)>0$ for every $s>0$;
\end{itemize}
or 
\begin{itemize}
\item[($f2$)] there exists (a unique) $s_f>0$ such that $f(s_f)=0$; then $f>0$ in $(0,s_f)$ and $f<0$ in $]s_f,+\infty[$.
\end{itemize}

The typical example of nonlinearity satisfying ($f1$) is $f(s)=|s|^{p-1}s$ for $0<p<1$, while $f(s)=\lambda(u-|u|^{p-1}u)$ satisfies ($f2$) for $p>1$.

We will be able to prove unified theorems under ($f1$) and ($f2$), however the strategy for both cases will be, in some situations, slightly different.

\smallbreak

Under condition ($f1$), (A1),(A4), we have
\[
|f(s)|\leq \hat{C}+(1-2\eps)\lambda_1(\Omega)|s| \ \  \text{ for every $s\in \R$}.
\]
for some $0<\eps<1/2$ and $\hat{C}>0$. In this case 
\begin{align}
|F(s)|\leq \hat{C}|s|+\frac{(1-2\eps)}{2}\lambda_1(\Omega) s^2 \leq C+\frac{(1-\eps)}{2}\lambda_1(\Omega) s^2 \quad \text{ for every $s\in \R$}.  \label{eq:bound_on_F}
\end{align}
In particular, the functional $I$ defined in \eqref{eq:functional_I} is always finite and $I(u)=\frac{1}{2}\int_\Omega |\nabla u|^2-\int_\Omega F(u)$ for every $u\in H^1_0(\Omega)$. Thus $I$ is of class $C^1$.

\medbreak

Assume now (and for the rest of this section) that $f$ satisfies ($f2$). Then $I(u)$ might not be finite for some $u\in H^1_0(\Omega)$, and in the proofs of some results we will rely on the following truncation of the nonlinearity $f$:
\begin{equation}\label{eq:tilde f}
\tilde f(t):=\begin{cases}
f(t) & -s_f \leq t\leq s_f,\\
0 & t \not\in [-s_f,s_f].
\end{cases}
\end{equation}
We remark that $f(t)\leq \tilde f(t)$ for $t\geq 0$, so that 
\[
F(t)\leq \tilde F(t) \quad \text{ for every $t\in \R$.}
\] 
Observe also that $\tilde f$ satisfies (A1), (A3), (A4), but not (A2). We define the \emph{truncated} functional $\tilde I: H^1_0(\Omega)\to \R$ as
\[
\tilde I(u):=\frac{1}{2}\int_\Omega |\nabla u|^2-\int_\Omega \tilde F(u)
\]
which satisfies:
\begin{equation*}
\tilde I(u)\leq I(u) \qquad \forall u\in H^1_0(\Omega).
\end{equation*}
Moreover, since $\tilde f$ is bounded, then $|F(s)|\leq C|s|$ for some $C>0$, and in particular estimate \eqref{eq:bound_on_F} holds true. Observe that $\tilde I$ is a $C^1$ functional (even though $I$ might not be).

\begin{lemma}
Let $f$ be a function satisfying (A1)-(A2) and ($f2$). Let $u\in H^1_0(\Omega)$ be any bounded solution of either $-\Delta u=f(u)$ or $-\Delta u=\tilde f(u)$, where $\tilde f$ is defined as in  \eqref{eq:tilde f}. Then 
\[
-s_f\leq u(x)\leq s_f \quad \text{ for every $x\in \Omega$.}
\]
\end{lemma}
\begin{proof}
Let $u$ be a solution of $-\Delta u=g(u)$, where $g$ is either equal to $f$ or to $\tilde f$. Testing the equation with $(u-s_f)^+$ we obtain
\[
\int_\Omega |\nabla (u-s_f)^+|^2=\int_\Omega g(u)  (u-s_f)^+\leq 0
\]
and so $(u-s_f)^+\equiv 0$, that is, $u\leq s_f$. Testing with $(u+s_f)^-$, we obtain in a similar fashion that $-s_f\leq u$.
\end{proof}

As an immediate consequence we have the following.

\begin{corollary}\label{coro:consequences}
Let $f$ satisfy (A1)-(A4) and ($f2$).
\begin{itemize}
\item[(i)] Denoting by $w$ the unique positive bounded solution of \eqref{eq:mainproblem}, we have $\|w\|_\infty\leq s_f$;
\item[(ii)] Critical points of $\tilde I$ are solutions of \eqref{eq:mainproblem}.
\end{itemize}
\end{corollary}

Another important consequence is that the absolute minimizers of $I$ and $\tilde I$ coincide.

\begin{corollary}\label{coro:tilde_m}
Let $f$ satisfy (A1)-(A4) and ($f2$), and define 
\[
\tilde m:=\inf\{ \tilde I(v):\ v\in H^1_0(\Omega)\}.
\]
Then $m=\tilde m$. In particular, $\pm w$ are the unique global minimizers of $\tilde I$.
\end{corollary}
\begin{proof}
Since $I\leq \tilde I$ then $m\leq \tilde m$; on the other hand, since $\|w\|_{\infty}\leq s_f$, then $m=I(w)=\tilde I(w)\geq \tilde m$.
\end{proof}

\section{Sign-Changing solutions} \label{sec:nodalsolutions}

In order to prove the existence of sign-changing solutions, in this section we replace (A3) with the stronger assumption
\begin{itemize}
\item[(A3')] $\lim_{s\to 0^+} \frac{f(s)}{s}>\lambda_2(\Omega)$\,.
\end{itemize}
Hence, in what follows, we assume that $f$ satisfies (A1)-(A2)-(A3')-(A4). Recall that either ($f1$) or ($f2$) can happen, and that this will not affect the statements but simply the proofs.

We denote by $\varphi_1$ the first positive eigenfunction of the Dirichlet Laplacian, with $\int_\Omega \varphi_1^2=1$. Recall the following two characterizations of the second eigenvalue:
\begin{equation}\label{eq:characterizations_lambda2}
\lambda_2(\Omega)  = \inf_{\omega\subset \Omega} \max \{ \lambda_1(\omega),\lambda_1(\Omega\setminus\overline \omega ) \}	=\inf_{c\in \Lambda} \sup_{t\in [0,1]} \int_\Omega |\nabla c (t)|^2,
\end{equation}
where
\[
\Lambda=\left\{c\in C([0,1],H^1_0(\Omega)),\ c(0)=-\varphi_1,\ c(1)=\varphi_1,\ \|c(t)\|_{L^2}=1\ \forall t\right\}.
\]
(for a proof of these two characterizations, see \cite{CTV05} and \cite{CuestaFigueiredoGossez99} respectively).

\begin{remark}
It can be seen directly from the first characterization of $\lambda_2(\Omega)$ in \eqref{eq:characterizations_lambda2} that the problem $-\Delta u=\lambda(u-|u|^{p-1}u)$ does not have a sign-changing solution for $0\leq \lambda\leq \lambda_2(\Omega)$, and so condition (A3') is, in a sense, sharp.
\end{remark}

Recall we are denoting by $w$ be unique positive solution of \eqref{eq:mainproblem}, and that $\pm w$ are the unique global minimizers of $I$. Define the mountain-pass level
\begin{equation}\label{eq:cmp}
c_{mp}=\inf_{\gamma\in \Gamma} \sup_{t\in [0,1]} I(\gamma(t))
\end{equation}
where
\[
\Gamma=\{\gamma\in C([0,1],H^1_0(\Omega)):\ \gamma(0)=-w,\ \gamma(1)=w \}.
\]

The first main result of this section is the following.

\begin{theorem}\label{thm:MP}
Assume that $f$ satisfies (A1)-(A2)-(A3')-(A4).
There exists $u\in H^1_0(\Omega)$, a bounded solution of \eqref{eq:mainproblem}, such that $I(u)=c_{mp}$. Moreover, $m<c_{mp}<0$, and any $u\in H^1_0(\Omega)\cap L^\infty(\Omega)$ achieving $c_{mp}$ is a sign-changing solution of \eqref{eq:mainproblem}.
\end{theorem}

In the rest of the section, we assume (A1)-(A2)-(A3')-(A4) and therefore omit the reference to these assumptions. The proof of Theorem \ref{thm:MP} will be slightly different in case $f$ satisfies ($f2$); in such case we will rely on the truncation $\tilde I$ introduced in Section \ref{subsec:truncation}, from which we borrow all notations. In such case, we need the following alternative characterization of $c_{mp}$.

\begin{lemma}\label{lemma:tilde c}
Assume that $f$ satisfies ($f2$). Define
\[
\tilde c_{mp}:=\inf_{\gamma\in \Gamma} \sup_{t\in [0,1]} \tilde I(\gamma(t))\,.
\]
Then $c_{mp}=\tilde c_{mp}$.
\end{lemma}
\begin{proof}
Since $\tilde I\leq I$, it is clear that $\tilde c_{mp}\leq c_{mp}$.

As for the other inequality, let us introduce the transformation $T:H^1_0(\Omega)\to H^1_0(\Omega)$ by
\[
T(u):=\max\{-s_f,\min\{u,s_f\}\}=\begin{cases} s_f & \text{ if } u(x)\geq s_f\\   u(x) &\text{ if } -s_f\leq u(x)\leq s_f \\ -s_f & \text{ if } u(x)\leq -s_f\end{cases}
\]
From the definitions, we have directly that $\tilde I(T(u))=I(T(u))$. 

Moreover, $F(u)\leq F(T(u))$ for every $u\in H^1_0(\Omega)$, since:
\begin{itemize}
\item[] for $u(x)\in [-s_f,s_f]$, $T(u)(x)=u(x)$ and $F(u) = F(T(u))$;
\item[] for $u(x)\geq s_f$, $F(u)\leq F(s_f)=F(T(u))$, since $F'=f<0$ in $(s_f,\infty)$, thus decreasing;
\item[] for $u(x)\leq -s_f$ we have $F(u)\leq F(T(u))$ by the previous paragraph and since $F$ is even symmetric.
\end{itemize}
From this we have
\[
I(T(u))=\frac{1}{2} \int_\Omega|\nabla T(u)|^2 - \int_\Omega F(T(u))  \leq \frac{1}{2} \int_\Omega |\nabla u|^2 - \int_\Omega F(u)=I(u). 
\]
Moreover, observe that $T(\pm w)=\pm w$ by Lemma \ref{coro:consequences}-(i). 

In conclusion, given a path $\gamma\in \Gamma$, we have $T\circ \gamma \in \Gamma$, and
\[
\sup_{t\in [0,1]} I(\gamma(t)) \geq \sup_{t\in [0,1]} I(T\circ \gamma(t)) = \sup_{t\in [0,1]} \tilde I(T\circ \gamma(t)) \geq \tilde c_{mp};
\]
whence $c_{mp}\geq \tilde c_{mp}$
\end{proof}

\begin{lemma}\label{lemma:MP_min}
Let $0<\eps<2\|w\|_{H^1_0}$. Then
\begin{itemize}
\item[-] If ($f1$) holds, then $\inf\{  I(u):\ \|u-w\|_{H^1_0}=\eps\}>m.$
\item[-] If ($f2$) holds, then $\inf\{  \tilde I(u):\ \|u-w\|_{H^1_0}=\eps\}>\tilde m.$
\end{itemize}
\end{lemma}
\begin{proof}
We follow \cite{CuestaFigueiredoGossez99}.  
\smallbreak

Let us first consider the case where $f$ satisfies ($f1$). Suppose the conclusion does not hold. Then there exists $\{u_n\}\subset H^1_0(\Omega)$ such that
\[
 I(u_n)\to m,\qquad \|u_n-w\|_{H^1_0}=\eps.
\]
Thus, up to a subsequence, $u_n\rightharpoonup u$ in $H^1_0(\Omega)$, and in particular $\int_\Omega |\nabla u|^2\leq \liminf \int_\Omega |\nabla u_n|^2$. On the other hand, since \eqref{eq:bound_on_F} holds, then by dominated convergence and the compact embedding $H^1_0(\Omega)\hookrightarrow L^2(\Omega)$, we deduce that $\int_\Omega F(u_n)\to \int_\Omega F(u)$. Therefore,
\[
m \leq  I(u)\leq \lim_{n \to +\infty}  I (u_n)=m,
\]
and $u=w$ because $\eps<\text{dist}(w,-w)$ and $\pm w$ are the unique global minimizers of $I$. Combining now
\[
I(u_n)\to  I(w)\quad \text{ with } \quad \int_\Omega F(u_n)\to \int_\Omega F(w),
\]
we deduce that $\int_\Omega |\nabla u_n|^2\to \int_\Omega |\nabla w|^2$. Since $u_n\rightharpoonup w$ weakly in $H^1_0(\Omega)$, then actually the convergence is strong in $H^1_0(\Omega)$, which contradicts the assumption $\|u_n-w\|_{H^1_0}=\eps>0$.

\smallbreak

The case where $f$ satisfies ($f2$) follows exactly in the same way, replacing $I$ and $F$ respectively by $\tilde I$ and $\tilde F$, and recalling that $m=\tilde m$ and that $\pm w$ are the unique global minimizers of both $I$ and $\tilde I$ (see Lemma \ref{coro:tilde_m}). 
\end{proof}

\begin{lemma}
We have $c_{mp}<0$.
\end{lemma}
\begin{proof}
1) Recalling Remark \ref{rem_energynegative}, we can choose $\varepsilon>0$ small such that $I(\varepsilon \varphi_1)<0$. Recall also that $I(u)=m<0$ by Lemma \ref{lemma:m_achieved}. Consider the path		

		\[
	\gamma:[0,1]\to H^1_0(\Omega),\qquad	\gamma(t)=\sqrt{(1-t) w^2+t(\eps \vphi_1)^2}
		\]
		which links $\gamma(0)=w>0$ to $\gamma(1)=\eps \vphi_1 >0$. Condition (A2) implies that $t\mapsto F(\sqrt{t})$ is strictly concave. On the other hand, it is proved in \cite[Lemma 3.9]{BFSST} that $t\mapsto \int_\Omega |\nabla \gamma(t)|^2$ is strictly convex. Then $t\mapsto I(\gamma(t))$ is stricly convex, and in particular
		\[
		I(\gamma(t))\leq (1-t)I(w)+tI(\eps\vphi_1)<0 \qquad \forall t\in [0,1].
		\]

	\smallbreak
		
\noindent 2) Assumption (A3') implies the existence of $\bar s,\delta>0$ such that $F(s)>\frac{1+\delta}{2}\lambda_2(\Omega)s^2$ for every $|s|\leq \bar s$.  From the second characterization of $\lambda_2(\Omega)$ presented in \eqref{eq:characterizations_lambda2}, we can take a continuous path $c$, joining $-\vphi_1$ to $\vphi_1$, such that
\[
\int_\Omega |\nabla c(t)|^2\leq \lambda_2(\Omega)(1 + \frac{\delta}{2}),\qquad \int_\Omega [c(t)]^2=1 \qquad \forall t \in [0,1].
\] 
By eventually choosing a smaller $\varepsilon>0$ from the start, we can assume that $\|\varepsilon c(t)\|_\infty \leq \bar s$ for every $t\in [0,1]$. Thus
\[
I(\varepsilon c(t))\leq \frac{\varepsilon^2}{2}\int_\Omega |\nabla c(t)|^2 - \int_\Omega F(\varepsilon c(t))\leq  -\frac{\delta}{4} \lambda_2(\Omega) \varepsilon^2<0.
\]
\smallbreak

\noindent 3) By considering the paths $-\gamma$, $c$, and $\gamma$ (in this order), we can join $-w\to -\varepsilon \varphi_1\to \varepsilon \varphi_1\to w$  with a continuous curve along which $I<0$. This implies the statement made in the lemma.
\end{proof}

\begin{proof}[Proof of Theorem \ref{thm:MP}]
Assume that $f$ satisfies ($f1$). Then $I$ is of class $C^1$ and from \ref{eq:bound_on_F} it is standard to check that $I$ satisfies the Palais-Smale condition. This combined with Lemma \ref{lemma:MP_min} allows to apply the Mountain Pass Theorem, and the $ c_{mp}$ is critical for $I$. Since $0> c_{mp}\geq \inf\{  I(w):\ \|u-w\|=\eps\}> m$ and $\pm w$ are the unique signed solutions of \eqref{eq:mainproblem}, then any critical point achieving $c_{mp}$ is necessarily sign-changing.

Assume now that $f$ satisfies ($f2$). Then the proof follows in the same way by replacing $I$ and $F$ by $\tilde I$ and $\tilde F$ respectively, and using Corollary \ref{coro:tilde_m},  Lemma \ref{lemma:tilde c} and Lemma \ref{lemma:MP_min}.
\end{proof}

We proved the existence of at least one sign-changing solution. Let us define the least energy nodal level as
\begin{equation}\label{eq:cnod}
c_{nod}=\inf\{I(u):\ u\in H^1_0(\Omega)\cap L^\infty(\Omega), u^\pm\not\equiv 0,\ -\Delta u=f(u)\}.
\end{equation}
Clearly
\[
-\infty<m< c_{nod}\leq c_{mp}<0.
\]
In the following we prove the existence of a least energy nodal solutions.

\medbreak

\medbreak

\begin{theorem}\label{thm:lensachieved}
Assume that $f$ satisfies (A1)-(A2)-(A3')-(A4). Then the level $c_{nod}$ is achieved. In particular, $m<c_{nod}$.
\end{theorem}
\begin{proof}
By Lemma \ref{lemma:uniformbounded}, we know that every bounded solution $u$ satisfies $-w\leq u\leq w$ a.e. in $\Omega$. Take a minimizing sequence $\{u_n\}$:
\[
u_n^\pm\not\equiv 0,\qquad -\Delta u_n=f(u_n),\qquad I(u_n)\to c_{nod}.
\]
Thus $ \|u_n\|_\infty \leq \|w\|_\infty$, and also $\Delta u_n$ is uniformly bounded in $L^\infty(\Omega)$-norm. Hence, by standard elliptic regularity theory, there exists $u$ such that, up to a subsequence,
\[
u_n\to u \text{ in } C^{1,\alpha}(\overline \Omega),\qquad \text{ with } \qquad -\Delta u=f(u).
\]
If $u^\pm\not \equiv 0$, then we are done. Suppose, by contradiction, that $u^-\equiv 0$. Since $c_{nod}<0$, then necessarily $u^+\not\equiv 0$. We have $f(u)\geq 0$: if ($f1$) is satisfied it is obvious, while in the case of ($f2$) it follows from the fact that $0\leq u \leq s_f$. So, by the maximum principle, $u>0$ in $\Omega$. In conclusion, $u=w$, the unique positive solution of \eqref{eq:mainproblem}, is the $C^{1,\alpha}(\overline \Omega)$--limit of sign-changing functions which are zero on the boundary. 

Assume $\Omega$ is a $C^{1,1}$-- domain. By Hopf's lemma, $\frac{\partial w}{\partial \nu}<0$ on $\partial \Omega$. Thus there exist $\varepsilon,\tilde \varepsilon, \delta>0$ such that
\begin{align}
&|\nabla w|\geq 2\varepsilon  &\forall x:\ \text{dist}(x,\partial \Omega)\leq 2\delta\\
&w\geq 2 \tilde \varepsilon     &\forall x:\ \text{dist}(x,\partial \Omega)\geq 2\delta
\end{align}
and so, if $u_n\to u$ in $C^{1,\alpha}(\overline \Omega)$, then for large $n$
\begin{align}
&|\nabla u_n|\geq \varepsilon  &\forall x:\ \text{dist}(x,\partial \Omega)\leq 2\delta\\
&u_n\geq  \tilde \varepsilon &\forall x:\ \text{dist}(x,\partial \Omega)\geq 2\delta.
\end{align}
If $u_n$ is sign-changing, then since $u_n=0$ on $\partial \Omega$, necessarily $u_n$ achieves a minimum on $x_n$ such that $0<d(x_n,\partial \Omega)\leq 2\delta$, and $\nabla u_n(x_n)=0$, which is a contradiction. 
\end{proof}

\begin{remark}
The proof of Theorem \ref{thm:lensachieved} strongly relies on regularity assumptions on the boundary of $\Omega$. This allow the use of Hopf's lemma, used to prove that in a $C^{1}$-neighborhood of the (unique) positive solution there are no other solutions of the problem. One might wonder whether the result still holds true for a general bounded domain with Lipschitz boundary. In the particular case of the nonlinearity $f(s)=|s|^{p-1}s$ with $0<p<1$, it is proved in \cite{BrascoDePhilippisFranzina}  that the positive solution is isolated in $L^1$-sense, but only for a certain range of exponents which depends on the geometry of the boundary. If $\Omega$ is a $C^1$-domain, they could prove that the same result holds true for every $p \in (0,1)$.  In these situations, this implies the existence of least-energy nodal  solutions. \end{remark}

\section {A domain where $c_{nod}\neq c_{mp}$} \label{sec:nodaldumbel}

\begin{theorem}\label{thm:cmp_neqc_nod}
Suppose that $f$ satisfies (A1), (A2) and that
\begin{equation}
  \label{eq:assumption-f-1}
\lim_{s \to 0}\frac{f(s)}{s}>0  
\end{equation}
and
\begin{equation}
  \label{eq:assumption-f-2}
\lim_{s \to \infty}\frac{f(s)}{s} \le 0.  
\end{equation}
Then there exist domains such that $c_{nod} < c_{mp}$,
which implies the existence of a least energy nodal solution which is not of mountain pass type.
\end{theorem}

\begin{remark} Observe that condition \eqref{eq:assumption-f-2} implies that (A4) is satisfied for every domain $\Omega$.
\end{remark}

The rest of the section is devoted to the proof of this result. The domain for which $c_{nod}<c_{mp}$ will be a dumbbell with two sufficiently large balls connected by a tube of sufficiently small width. 

Let $B_1$ and $B_2$ be disjoint open balls with common radius $r>1$, which is chosen large enough so that 
\begin{equation}
  \label{eq:assumption-f-1-proof}
\lim_{s \to 0}\frac{f(s)}{s}>\lambda_1(B_1)=\lambda_1(B_2).  
\end{equation}
For $\delta \in (0,1]$, let $\Omega_\delta$ be the dumbbell domain obtained by connecting the centers of $B_1$ and $B_2$ with a tube of width $\delta \in (0,1)$. Let $\Omega_*$ be the convex hull of $B_1$ and $B_2$, which contains all the sets $\Omega_\delta$, $0< \delta  < 1$ as well as $\Omega_0 := B_1 \cup B_2$. By trivial extension, we will consider $H^1_0(\Omega_\delta)$ as a subspace of $H^1_0(\Omega_*)$ for every $\delta \in [0,1)$, and we consider the functional 
$$
I: H^1_0(\Omega_*) \to \R, \qquad  
I(u)=\begin{cases}
\frac{1}{2}\int_{\Omega^*} |\nabla u|^2-\int_{\Omega^*} F(u) & \text{ if } F(u)\in L^1(\Omega_*),\\
+\infty & \text{ if } F(u)\not \in L^1(\Omega_*).
\end{cases}
$$
Observe that $f$ satisfies (A1)-(A4) for every domain $\Omega_\delta$, $\delta\in [0,1)$.
If $f$ satisfies condition ($f1$), recall that $I$ is of class $C^1$ and $I(u)=\frac{1}{2}\int_{\Omega_*} |\nabla u|^2-\int_{\Omega_*} F(u)$ for every $u\in H^1_0(\Omega_*)$, while for ($f2$) we will rely on the truncation $\tilde I$ introduced in Section \ref{subsec:truncation}.

Let $w_\delta$ be the unique positive solution of \eqref{eq:mainproblem} in $\Omega_\delta$, which satisfies $I(w_\delta)=m(\Omega_\delta)$ for $\delta \in [0,1)$. Moreover, let $w_1$ and $w_2$ the positive solutions in $B_1$ and $B_2$ respectively, so that $w_2$ is a mere translation of $w_1$ and $w_0 = w_1 +w_2$. We also have that  
$$
m(\Omega_0)= I(w_0)=I(w_1)+ I(w_2)=2 I(w_1)= 2 m(B_1).
$$

\begin{lemma}\label{lemma_convergence}
We have, as $\delta\to 0^+$,
\[
m(\Omega_\delta)\to m(\Omega_0),\quad \text{ and } \quad  w_\delta \to w_0 \text{ strongly in } H^1_0(\Omega_*).
\]
\end{lemma}
\begin{proof}
First of all, since $w_0 \in H^1_0(\Omega_0) \subset H^1_0(\Omega_\delta)$ for $\delta \in (0,1)$, we have that 
\begin{equation}
  \label{eq:prelim-lim-m_delta}
m(\Omega_\delta) \le I(w_0)= m(\Omega_0) \qquad \text{for $\delta \in [0,1]$.}
\end{equation}

\smallbreak

Assume first that $f$ satisfies condition ($f1$). Then \eqref{eq:assumption-f-2} implies the existence of $C,\eps>0$ such that 
\begin{equation}\label{eq:upperbound_F}
|F(s)|\leq C+\frac{(1-\eps)}{2}\lambda_1(\Omega_*) s^2 \quad \text{ for every $s\in \R$},
\end{equation}
which implies that
\begin{align*}
m(\Omega_0) \geq I(w_\delta) \geq \frac{1}{2}\int_{\Omega_*} |\nabla w_\delta|^2 - \frac{1-\eps}{2}\lambda_1(\Omega_*) \int_{\Omega_*} w_\delta^2 - C|\Omega_*|= \frac{\eps}{2}\int_{\Omega_*} |\nabla w_\delta|^2-C|\Omega_*|.
\end{align*}
Therefore, the functions  $w_\delta$, $\delta \in [0,1)$, are uniformly bounded in $H^1_0(\Omega_*)$ and there exists a function $\hat w$ such that, up to a subsequence, $w_\delta \rightharpoonup \hat{w}$ in $H^1_0(\Omega_*)$ and  pointwise a.e. in $\Omega$. The latter pointwise convergence implies that $\hat{w} \equiv 0$ a.e. in $\Omega_* \setminus \Omega_0$. Since $ \Omega_0$ is a Lipschitz domain, then $\hat w\in H^1_0(\Omega_0)$. Moreover, again by \eqref{eq:upperbound_F} and by Lebesgue's Dominated Convergence Theorem, we have
\[
\int_\Omega F(w_\delta)\to \int_\Omega F(\hat w) <\infty
\]
Combining this with \eqref{eq:prelim-lim-m_delta} we have
$$
m(\Omega_0) \le I(\hat{w}) \leq \liminf_{\delta \to 0} I(w_\delta) 
\le \limsup_{\delta \to 0} I(w_\delta) \le m(\Omega_0).
$$
Hence all inequalities are equalities. From this we deduce $m(\Omega_\delta)\to m(\Omega_0)$, $\hat w=w_0$, and also the convergence 
$$
\int_{\Omega_*}|\nabla w_\delta|^2\,dx \to \int_{\Omega_*}|\nabla w_0|^2\,dx;
$$
which combined with the weak convergence $w_\delta\rightharpoonup w_0$ in $H^1_0(\Omega_*)$ yields
\begin{equation}
  \label{eq:strong-convergence-u-delta}
w_\delta \to w_0\qquad \text{strongly in $H^{1}_0(\Omega_*)$.}  
\end{equation}
\smallbreak

Assume now that $f$ satisfies condition ($f2$). Then we can repeat the previous argument simply replacing $F$ and $I$ respectively by $\tilde F$ and $\tilde I$, recalling also that $-s_f\leq w_\delta(x)\leq s_f$ for every $x\in \Omega$, $\delta\in (0,1)$.
\end{proof}

Let us now define 
$$
\eps_*:= \|w_1\|_{H^1(B_1)}= \|w_2\|_{H^1(B_2)}
$$
and $W := w_1 - w_2 \in H^1_0(\Omega_0)$, which satisfies $I(W) = I(w_0)$. Then we have that 
\begin{equation}
  \label{eq:strong-convergence-u-delta-1}
\|w_\delta - W\|_{H^{1}}  \to \|w_0 - W\|_{H^{1}} = 2\eps_* \qquad \text{as $\delta \to 0$.}
\end{equation}
By Lemma \ref{lemma_convergence} and \eqref{eq:strong-convergence-u-delta-1}, we may choose $\delta_0>0$ such that 
\begin{equation}
  \label{eq:u-delta-w-dist}
\|w_\delta\|_{H^1}> \varepsilon_* \quad \text{and}\quad 
\|w_\delta - W\|_{H^{1}} > \varepsilon_* \qquad \text{for every $\delta < \delta_0$.}
\end{equation}

\begin{lemma}
Under the previous notations, there exists $\delta_1  \in (0,\delta_0)$ and $c > m(\Omega_0)$ such that
\begin{itemize}
\item If ($f1$) holds, then
\begin{equation}
  \label{eq:energ-barrier}
 \text{$I(v) \geq c$ for every $\delta < \delta_1$ and every $v \in H^{1}_0(\Omega_{\delta})$ with $\|v - w_\delta\|_{H^{1}}=\varepsilon_{*}$.}
\end{equation}
\item If ($f2$) holds, then
\begin{equation}
  \label{eq:energ-barrier2}
 \text{$\tilde I(v) \geq c$ for every $\delta < \delta_1$ and every $v \in 
H^{1}_0(\Omega_{\delta})$ with $\|v - w_\delta\|_{H^{1}}=\varepsilon_{*}$.}
\end{equation}

\end{itemize}
\end{lemma}

\begin{proof}
Assume $f$ satisfies ($f1$); the other situation is analogous working with truncations. If \eqref{eq:energ-barrier} was not true, there would exist a sequence $\delta_k \to 0$ and functions $v_k \in H^{1}_0(\Omega_{\delta_k})$ with $\|v_k - w_{\delta_k}\|_{H^1}=\varepsilon$ for all $k$ and 
$$
I(v_k) \to m(\Omega_0) \qquad \text{as $k \to \infty$.}
$$
Since the sequence is bounded in $H^{1}_0(\Omega_*)$, one can extract a subsequence (still denoted by $v_k$) such that $v_k \rightharpoonup v$ weakly in $H^{1}_0(\Omega_*)$ and pointwise a.e. in $\Omega_*$. As in the proof of Lemma \ref{lemma_convergence}, we then infer that 
$v \in H^{1}_0(\Omega_0)$, and, by the weak lower semicontinuity of $I$, that 
$$
I(v) \leq \liminf_{k \to \infty}I(v_k) = m(\Omega_0).
$$
By definition of $m(\Omega_0)$, equality follows, and as before we then 
deduce that 
$$
\int_{\Omega_*}|\nabla v_k|^2\,dx \to \int_{\Omega_*}|\nabla v|^2\,dx \qquad \text{as $k \to \infty$,}
$$
and $v_k \to v$ strongly in $H^{1}(\Omega_*)$. This together with Lemma \ref{lemma_convergence} implies that 
$$
\|v - w_0\|_{H^1} = 
\varepsilon_*.
$$
On the other hand, since $I(v)=m(\Omega_0)$, we have that either $v=w_0$, $v=-w_0$ or $v = W$, which implies that
either $\|v - w_0\|_{H^1}=0$, $\|v - w_0\|_{H^1}=2 \|w_0\|_{H^1}= 4 \varepsilon_*$ or $\|v - w_0\|_{H^1}=2 \eps_*$. Hence all cases are impossible, which gives a contradiction. 
\end{proof}
By the evenness of $I$, Lemma \ref{lemma_convergence} also implies that  
\begin{itemize}
\item If ($f1$) holds, then
\begin{equation}
  \label{eq:energ-barrier-1}
 \text{$I(v) \geq c$ for every $\delta < \delta_1$ and every $v \in H^{1}_0(\Omega_{\delta})$ with $\|v + w_\delta\|_{H^{1}}=\varepsilon_{*}$.}
\end{equation}
\item If ($f2$) holds, then
\begin{equation}
  \label{eq:energ-barrier-2}
 \text{$\tilde I(v) \geq c$ for every $\delta < \delta_1$ and every $v \in H^{1}_0(\Omega_{\delta})$ with $\|v + w_\delta\|_{H^{1}}=\varepsilon_{*}$.}
\end{equation}
\end{itemize}

\begin{proof}[Conclusion of the proof of Theorem \ref{thm:cmp_neqc_nod}]
Again, we do it only in case ($f1$), since for ($f2$) is analogous using truncations. We now fix $\delta \in (0,\delta_1)$, and we prove the claim of the theorem
for the domain $\Omega_\delta$. For this we define 
$$ 
\mathcal{A_\delta}:=\{ v \in H^{1}_0(\Omega_\delta)\,|\,\|v\pm w_\delta \|_{H^1(\Omega_\delta)}> \varepsilon_* \},
$$
and we note that 
\begin{equation}
  \label{eq:energy-barrier-2}
\inf_{\partial \mathcal{A_\delta}} I \ge c  
\end{equation}
by \eqref{eq:energ-barrier} and \eqref{eq:energ-barrier-1}. Moreover, since $\|w_\delta\|_{H^1} >\eps_*$ by \eqref{eq:u-delta-w-dist}, every path joining $w_\delta$ and $-w_\delta$ intersects $\partial \mathcal{A_\delta}$. Consequently, 
\begin{equation}
  \label{eq:mountain-pass-barrier}
c_{mp} \ge c.  
\end{equation}
On the other hand, since $W \in \mathcal{A_\delta}$ by \eqref{eq:u-delta-w-dist}, we have 
\begin{equation}
  \label{eq:c-1-est}
\hat{c}:=\inf_{\mathcal{A_\delta}} I \le I(W) = m(\Omega_0)< c.
\end{equation}
By \eqref{eq:energy-barrier-2}, \eqref{eq:c-1-est} and a standard application of Ekeland's variational principle, we then find a sequence $\{v_k\}$ in $\mathcal{A_\delta}$ such that 
$$
I(v_k)\to \hat{c} \qquad \text{and}\qquad
\|I'(v_k)\|_{H^1_0(\Omega_\delta)^*} \to 0.
$$
Since the functional $I|_{H^1_0(\Omega_\delta)}$ satisfies the Palais-Smale condition, it is possible to extract a subsequence -- still denoted by $\{v_k\}$ -- such that $v_k \to v$ in $H^{1}_0(\Omega_\delta)$. Therefore, 
$I(v)=c_1$, which means that $v$ is a local minimizer of $I|_{H^1_0(\Omega_\delta)}$. By construction, $v$ is thus a nodal solution on $\Omega_\delta$, so that 
$$
c_{nod} \le \hat{c} < c \le c_{mp},
$$
as claimed.
\end{proof}

\begin{remark}
A further sign changing solution can be found by applying the Mountain-Pass theorem to the class of paths joining $w_\delta$ with the solution $v$ found in the previous theorem. Moreover, if, in addition, we assume that $f \in C^1(\R)$, then we can find two further sign changing critical points $u_1,u_2$, where
$$
-w_\delta \le u_1 \le v \le u_2 \le w_\delta.
$$
This follows by applying a suitable variant of the Mountain-Pass theorem in order intervals. Under somewhat different assumptions, this Mountain-Pass theorem can be found e.g. in \cite[Theorem 1.3]{Li-Wang}. The main underlying tool needed in our setting is Lemma~\ref{sec:nodal-solut-morse} ahead. For the sake of brevity, we omit the details. 
\end{remark}

\section{Nodal solutions, Morse index and symmetry} \label{sec:morseindex}

Recall assumptions (A1)-(A4) from Section \ref{section:positivesolutions} and $(A3')$ from Section \ref{sec:nodalsolutions}. Throughout this section we assume in addition that $f\in C^1(\R)$. In this case we can define the Morse index $m(u)$ of a solution $u$ of \eqref{eq:mainproblem} as the number of negative Dirichlet eigenvalues of the operator $-\Delta - f'(u)$ in $\Omega$ (counted with multiplicity). We start with the following Morse index estimate for least energy nodal solutions.

\begin{theorem}\label{prop:boundmorseindex}
Assume that $f$ satisfies (A1)-(A2)-(A3')-(A4), and let $u \in H^1_0(\Omega)$ be a bounded nodal solution of \eqref{eq:mainproblem} with $I(u)=c_{nod}$. Then $m(u)\leq 1$.
\end{theorem}

The proof of this theorem and other results in this section uses global compactness and invariance properties of an $H^1$-gradient flow associated with the functional $I$ in a suitable subspace of $H^1_0(\Omega)$. Under slightly different assumptions, these invariance properties are derived in \cite[Sections 3 and 7]{bartsch}. For matters of completeness, we  include the derivation here under the present assumptions, essentially following the arguments in \cite{bartsch} in our proof of Lemma~\ref{sec:nodal-solut-morse} ahead. 

We recall that $w$ denotes the unique bounded positive solution of (\ref{eq:mainproblem}), and we fix $\kappa>0$ with the property that 
\begin{equation}
  \label{eq:strict-increase}
t \mapsto g(t):= f(t)+ \kappa t \qquad \text{is strictly increasing in the interval $\bigl[-\|w\|_{L^\infty}, \|w\|_{L^\infty}\bigr]$.}
\end{equation}

Within this section, we consider the equivalent scalar product
$$
(u,v) \mapsto \langle u,v \rangle_{H^1}:= \int_{\Omega}\Bigl( \nabla u \cdot \nabla v + \kappa uv \Bigr)\,dx 
$$
in $H^1_0(\Omega)$. Moreover, we consider the Banach space 
$$
C^1_0(\Omega):= \{u \in C^1(\overline \Omega)\::\: u\big|_{\partial \Omega}\equiv 0\} \subset H^1_0(\Omega),
$$
equipped with the usual norm $u \mapsto \|u\|_{C^1}:= \|u\|_{L^\infty} + \|\nabla u\|_{L^\infty}$. We note that, since $f \in C^1(\R)$, the restriction of the functional $I$ to the space $C^1_0(\Omega)$ is of class $C^2$. Moreover, for a function $u \in C^1_0(\Omega)$, the gradient of $I$ at $u$ with respect to $\langle \cdot , \cdot \rangle_{H^1}$ is given by $u - K(u)$, where 
\[
K: C^1_0(\Omega) \to C^1_0(\Omega), \qquad \qquad K(u)=(-\Delta +\kappa)^{-1}g(u).
\]
In other words, $v= K(u) \in C^1_0(\Omega)$ is the unique solution of the Dirichlet problem 
$$
-\Delta v + \kappa v = g(u) \quad \text{in $\Omega$,} \qquad \qquad v = 0 \quad \text{on $\partial \Omega$,}
$$
and we have 
\begin{equation}
  \label{eq:derivative-formula-restricted}
I'(u)z = \int_{\Omega} \nabla u \nabla z \,dx - \int_{\Omega}f(u)z\,dx = 
\langle u- K(u), z \rangle_{H^1} \qquad \text{for $u, z \in C^1_0(\Omega)$.}
\end{equation}
The following is a well known and straightforward consequence of classical elliptic estimates and the assumption $f \in C^1(\R)$.

\begin{lemma}
\label{sec:nodal-solut-morse-1}
\begin{enumerate}
\item[(i)] $K: C^1_0(\Omega) \to C^1_0(\Omega)$ is locally Lipschitz continuous. 
\item[(ii)] If $A \subset C^1_0(\Omega)$ is bounded with respect to $\|\cdot\|_{L^\infty(\Omega)}$, then $K(A)$ is relatively compact in $C^1_0(\Omega)$. 
\end{enumerate}
\end{lemma}

Next, for $u\in C^1_0(\Omega)$, we let $t \mapsto \eta^t(u)$ be the unique solution of the initial value problem 
\begin{equation}\label{eq:eta_def}
\begin{cases}
\frac{d}{dt}\, \eta^t(u)=K(\eta^t(u))-\eta^t(u)\\
\eta^0(u)=u,
\end{cases}
\end{equation}
defined on $[0,\tau_m(u))$, where $\tau_m(u)$ is the maximal time of existence. Moreover, for functions $v_1,v_2 \in C^1_0(\Omega)$, we let 
$$
[v_1,v_2]:= \{\psi \in C^1_0(\Omega) \::\: v_1 \le \psi \le v_2 \quad \text{in $\Omega$}\}
$$
denote the {\em order interval} in $C^1_0(\Omega)$ spanned by $v_1$ and $v_2$.
We need the following lemma.

\begin{lemma}
\label{sec:nodal-solut-morse}  
Let $v_1,v_2 \in [-w,w]$ be solutions of (\ref{eq:mainproblem}) with $v_1 \le v_2$, and let $u \in [v_1,v_2]$. Then we have: 
\begin{enumerate}
\item[(i)] $\tau_m(u)= \infty$. 
\item[(ii)] $\eta^t(u) \in [v_1,v_2]$ for all $t \in [0,\infty)$. 
\item[(iii)] The map $t \mapsto I(\eta^t(u))$ is nonincreasing in $[0,\infty)$. Moreover, if $u$ is no solution of (\ref{eq:mainproblem}), then $I(\eta^t(u)) < I(u)$ for $t>0$.
\item[(iv)] For any sequence of numbers $t_n \ge 0$ with $t_n \to \infty$ we have $\eta^{t_n}(u) \to u_\infty$ in $C^1_0(\Omega)$ after passing to a subsequence, where $u_\infty \in [v_1,v_2]$ is a solution of (\ref{eq:mainproblem}).
\end{enumerate}
\end{lemma} 

\begin{proof}
As mentioned above, we essentially follow arguments in \cite{bartsch}. Let $C:= [v_1,v_2]$. Then $C \subset C^1_0(\Omega)$ is a closed convex set. Moreover, $K(C) \subset C$. To see this, we let $q \in C$ and $v = K(q)$. Then $z:= v-v_1$ satisfies 
\begin{equation}
\label{invariance-equation-1}
-\Delta z +\kappa z = g(q)-g(v_1)=c(x)(q-v_1) \quad \text{in $\Omega$,}\qquad \qquad z = 0  \quad \text{on $\partial \Omega$,}
\end{equation}
where $c(x) = \int_0^1 g'((1-s)v_1(x)+s q(x))ds$ for $x \in \Omega$. Since $g \in C^1(\R)$ and $v_1,v_2$ are bounded, we have $c \in L^\infty(\Omega)$. Moreover, $c \ge 0$ in $\Omega$ as a consequence of (\ref{eq:strict-increase}), and $q-v_1\geq 0$ in $\Omega$. Multiplying (\ref{invariance-equation-1}) with $-z^-= \min \{z,0\}\leq 0$ and integrating, we find that 
$$
\int_{\Omega}\bigl(|\nabla z^-|^2 + \kappa |z^-|^2\bigr)dx \le 0
$$
and therefore $z^- \equiv 0$, i.e., $v \ge v_1$. Similarly, we see that $v \le v_2$, and therefore $v \in C$.\\
Now let $u \in C$, and let $\cO(u):= \{\eta^t(u)\::\: t \in [0,\tau_m(u)) \} \subset C^1_0(\Omega)$ denote the trajectory starting at $u$. Since $K(C) \subset C$, a standard polygonal approximation of $\cO(u)$ (see e.g. \cite{Deimling}) yields that $\cO(u) \subset C$. 
From this and Lemma~\ref{sec:nodal-solut-morse-1}(ii), it follows that the set $K(\cO(u))$ is relatively compact in $C^1_0(\Omega)$. Next we note that, since $\frac{d}{dt} \, e^t \eta^{t}(u)= e^{t}K(\eta^t(u))$, we have 
$$
\eta^t(u)= e^{-t}\Bigl(u + \int_0^t e^{s}K(\eta^s(u))\,ds\Bigr) = e^{-t}u + \int_0^t e^{-s}\,K(\eta^{t-s}(u))\,ds.
$$
From this and the relative compactness of $K(\cO(u))$, we deduce that $\cO(u) \subset C^1_0(\Omega)$ is relatively compact as well, and this implies, in particular, that $\tau_m(u)=\infty$. Next we note that, by (\ref{eq:derivative-formula-restricted}),  
$$
\frac{d}{dt} I(\eta^t(u)) = I'(\eta^t(u))\bigl(K(\eta^t(u))-\eta^t(u)\bigr) = 
- \|K(\eta^t(u))-\eta^t(u)\|_{H^1}^2 \le 0 \qquad \text{for $t \in [0,\infty)$.}$$
Consequently, the map $t \mapsto I(\eta^t(u))$ is nonincreasing. Moreover, if $u$ is no solution of (\ref{eq:mainproblem}), then we have $K(u) \not = u$ and therefore  
$$
\frac{d}{dt} I(\eta^t(u)) = - \|K(\eta^t(u))-\eta^t(u)\|_{H^1}^2 < 0 \qquad \text{for sufficiently small $t \ge 0$.}
$$
Consequently, $I(\eta^t(u))<I(u)$ for $t >0$ in this case. Finally, since $\cO(u)$ is relatively compact in $C^1_0(\Omega)$, we have 
$$
c_u:= \inf \limits_{\cO(u)}I > - \infty.
$$
Therefore, if $(t_n)_n$ is a sequence of numbers $t_n \ge 0$ with $t_n \to \infty$, we have $\lim \limits_{n \to \infty} I(\eta^{t_n}(u)) = c_u>-\infty$.
Hence there exists a further sequence $(s_n)_n$ of numbers $0 \le s_n \le t_n$ with $t_n -s_n \to 0$ as $n \to \infty$ and 
\begin{equation}
  \label{eq:I-prime-sequence}
\frac{d}{dt}\Big|_{t=s_n} I(\eta^t(u)) = - \|K(\eta^{s_n}(u))-\eta^{s_n}(u)\|_{H^1}^2 \to 0 \qquad \text{as $n \to \infty$.}
\end{equation}
Moreover, by the relative compactness of $\cO(u)$, we may pass to a subsequence, still denoted by $(s_n)_n$, with the property that 
$$
\eta^{s_n}(u) \to u_\infty \qquad \text{in $C^1_0(\Omega)$ as $n \to \infty.$}
$$
From this and (\ref{eq:I-prime-sequence}), we deduce that $K(u_\infty)= u_\infty$, and thus $u_\infty \in C$ is a solution of (\ref{eq:mainproblem}). 
Moreover, by the continuity of the flow $\eta$, we infer that 
$$
\eta^{t_n}(u) = \eta^{t_n-s_n}(\eta^{s_n}(u)) \to u_\infty \qquad \text{as $n \to \infty$.}
$$
The proof is thus finished. 
\end{proof}

\begin{proof}[Proof of Theorem~\ref{prop:boundmorseindex} (completed)]
We first note that
$$
-w < u < w \quad \text{in $\Omega$}\qquad \qquad \text{and}\qquad \qquad \frac{\partial w}{\partial \nu} < \frac{\partial u}{\partial \nu}< -\frac{\partial w}{\partial \nu}
\quad \text{on $\partial \Omega$.}
$$
This follows by applying the strong maximum principle and the Hopf boundary lemma to the functions $u \pm w$. Consequently, we have 
\begin{equation}
  \label{eq:innt-property}
u \in \innt([-w,w]),  
\end{equation}
where $\innt([-w,w])$ denotes the interior of the set $[-w,w]$ in $C^1_0(\Omega)$.

Arguing by contradiction, we now suppose that $m(u) \ge 2$. Then there exists linearly independent eigenfunctions $\phi_1, \phi_2 \in C^1_0(\Omega)$ of the eigenvalue problem
$$
\begin{cases}
-\Delta \phi_i - f'(u) \phi_i= \mu_i \phi_i &\text{in $\Omega$}\\
\phi_i = 0 &\text{on $\partial \Omega$}
\end{cases}
$$
corresponding to eigenvalues $\lambda_1<\lambda_2<0$. Moreover, we may assume that $\phi_1>0$ in $\Omega$. From this we may deduce that 
$$
I''(u)(v,v) < 0 \qquad \text{for every $v \in Y$,}
$$
where $Y$ denotes the span of $\phi_1$ and $\phi_2$. Consequently, for $\rho>0$ sufficiently small we have 
$$
I(z)<I(u) \qquad \text{for every $z \in Y_\rho$,}
$$
where $Y_\rho:= \{u+v :v \in Y \::\: \|v\|_{C^1} = \rho\}$. Moreover, by (\ref{eq:innt-property}) we can make $\rho$ smaller if necessary to guarantee that $Y_\rho \subset [-w,w]$. We now consider the sets
\begin{align*}
Y_\rho^+ &:= \{z \in Y_\rho\::\: \text{$\eta^t(z) \in [0,w]$ for some $t\ge 0$}\},\\ 
Y_\rho^- &:= \{z \in Y_\rho\::\: \text{$\eta^t(z) \in [-w,0]$ for some $t\ge 0$}\}. 
\end{align*}
Take $\varepsilon>0$ small so that $ -w \le u- \eps \phi_1 < u <  u + \eps \phi_1 \le w$ in $\Omega$.
We claim that $z_\pm := u \pm \varepsilon \phi_1 \in Y_\rho^{\pm}$. Indeed, since $z_+ \in [u,w]$, Lemma~\ref{sec:nodal-solut-morse} implies that  
\begin{equation}
  \label{eq:inclusion-flow-special}
\eta^t(z_+) \in [u,w]\qquad \text{for all $t > 0$.}  
\end{equation}
Moreover, by Lemma~\ref{sec:nodal-solut-morse}(iv), there exists a solution $w_+ \in [u,w]$ of (\ref{eq:mainproblem}) 
and a sequence of numbers $t_n \ge 0$ with $t_n \to \infty$ and $\eta^{t_n}(z_+) \to w_+$ in $C^1_0(\Omega)$, which implies that 
$$
I(u) >I(z_+) \ge I(w_+). 
$$
Since $u$ is a least energy sign changing solution of (\ref{eq:mainproblem}), $w_+$ may not change sign, and therefore $w_+$ coincides with $w$, the unique positive solution of (\ref{eq:mainproblem}). Since $w>0$ in $\Omega$ and $\frac{\partial w}{\partial \nu} <0$ on $\partial \Omega$, there exists an open neighborhood of $N$ of $w$ in $C^1_0(\Omega)$ which only contains positive functions. By continuity of the flow $\eta$, we thus infer that there exists $t>0$ with $\eta^t(z_+) \in [u,w] \cap N \subset [0,w]$, and therefore $z_+ \in Y_\rho^+$. In the same way, we see that $z_- \in Y_\rho^-$. Moreover, the sets $Y_\rho^\pm$ are relatively open subsets of $Y_\rho$. Indeed, if $z \in Y_\rho^+$, a similar argument as above shows that $\eta^t(z) \in N$ for some $t>0$. Since $N$ is open and the map $\eta^t(\cdot)$ is continuous, it follows that also $\eta^t(z') \in N \cap [-w,w]  \subset [0,w]$ for $z'$ sufficiently close to $z$ in the $C^1$-norm. Hence $Y_\rho^+$ is open, and similarly we see that $Y_\rho^-$ is open. Next we claim that 
\begin{equation}
  \label{eq:empty-intersection}
Y_\rho^+ \cap Y_\rho^- = \varnothing.
\end{equation}
Suppose by contradiction that $z \in Y_\rho^+ \cap Y_\rho^-$. Since the sets $[0,w]$ and $[-w,0]$ are flow invariant, there must exist $t>0$ with 
$$
\eta^t(z) \in [0,w] \cap [-w,0] = \{0\},
$$
which contradicts the fact that 
$$
I(\eta^t(z)) \le I(z)<I(u)= c_{nod} \le c_{mp}<0= I(0)  
$$
by Theorem~\ref{thm:MP} and Lemma~\ref{sec:nodal-solut-morse}. Hence (\ref{eq:empty-intersection}) is true. Since $Y_\rho$ is a (two-dimensional) circle in $Y$ and therefore connected, it now follows that 
$$
Y_\rho^0:= Y_\rho \setminus (Y_\rho^+ \cup Y_\rho^-) \not = \varnothing.
$$
Let $z \in Y_\rho^0$. By Lemma~\ref{sec:nodal-solut-morse}, we see that 
\begin{equation}
  \label{eq:inclusion-flow-special-1}
\eta^t(z) \in [-w,w] \qquad \text{for all $t > 0$,}  
\end{equation}
and there exists a solution $z_\infty \in C^1_0(\Omega)$ of (\ref{eq:mainproblem}) with $-w \le z_\infty \le w$ and a sequence of numbers $t_n \ge 0$ with $t_n \to \infty$ and $\eta^{t_n}(u) \to z_\infty$ in $C^1_0(\Omega)$,
which again implies that 
$$
I(u) >I(z) \ge I(z_\infty). 
$$
Since $u$ is a least energy sign changing solution of (\ref{eq:mainproblem}), $z_\infty$ may not change sign, and therefore we deduce 
$z_\infty \in \{\pm  w\}$. However, similarly as above, we then would have $\eta^t(z) \in N \cap [-w,w] \subset [0,w]$ or $\eta^t(z) \in -N \cap [-w,w] \subset [-w,0]$ for some $t>0$,  and therefore $z \in Y_\rho^+ \cup Y_\rho^-$. Contradiction. We thus have shown that $m(u) \le 1$, as claimed.
\end{proof}

\begin{theorem}
\label{sec:lens-mp}
Suppose that $f$ satisfies (A1)-(A2)-(A3')-(A4) and suppose there exists a nodal solution $u \in C^1_0(\Omega)$ of \eqref{eq:mainproblem} with $I(u)=c_{nod}$ and $m(u) =1$ (we already know that $m(u) \le 1$).\\
Then $u$ is of mountain pass type. More precisely, there exists a path 
$$
\gamma: [-1,1] \to [-w,w] \subset C^1_0(\Omega) \qquad \text{with}\quad \gamma(\pm 1) = \pm w,\quad \gamma(0)= u
$$
and such that $I \circ \gamma$ is maximized precisely at $0$. In particular, we have 
$$
c_{nod}=c_{mp}
$$
in this case.
\end{theorem}

\begin{proof}
Let $\phi_1 \in H^1_0(\Omega)$ be the (up to normalization unique) positive Dirichlet eigenfunction of the operator $-\Delta - f'(u)$ corresponding to its lowest eigenvalue, which is negative by assumption. We then fix $\eps \in (0,1)$ and define 
$$
\gamma_0: [-\eps,\eps] \to H^1_0(\Omega),\qquad \gamma_0(s)= u+s \phi_1.
$$
If $\eps>0$ is chosen sufficiently small, we have 
$$
I(\gamma_0(s)) < I(u)=c_{nod} \qquad \text{for $s \in [-\eps,\eps] 
\setminus \{0\}.$}
$$
Moreover, since (\ref{eq:innt-property}) holds for $u$ by the same argument as in the proof of Theorem~\ref{prop:boundmorseindex}, we may also assume that  
$$
-w \le u- \eps \phi_1 < u<  u + \eps \phi_1 \le w \qquad \text{in $\Omega$.}
$$
Let $\eta$ be the flow defined in (\ref{eq:eta_def}). By Lemma~\ref{sec:nodal-solut-morse}, $\eta^t$ is well defined as a map $[-w,w] \to [-w,w]$ for every $t>0$. We claim that 
\begin{equation}
  \label{eq:claim-convergence-w}
\eta^{t}(u + \eps \phi_1) \to w \qquad \text{in $C^1_0(\Omega)$ as $t \to \infty$.}  
\end{equation}
Suppose by contradiction that this is not the case. Then there exists $\delta>0$ and a sequence of numbers $t_n \ge 0$ with $t_n \to + \infty$ and 
\begin{equation}
  \label{eq:claim-convergence-w-contradiction}
\|\eta^{t_n}(u + \eps \phi_1)-w\|_{C^1} \ge \delta \qquad \text{for all $n \in \N$.}
\end{equation}
By Lemma~\ref{sec:nodal-solut-morse}(iv) we have, after passing to a subsequence, $\eta^{t_n}(u+ \eps \phi_1) \to u_\infty$ as $n \to \infty$, where $u_\infty \in [u,w]$ is a solution of (\ref{eq:mainproblem}) and 
$$
I(u_\infty) \le I(u + \eps \phi_1)< I(u)=c_{nod}< 0.
$$
Hence $u_\infty \not= 0$, and $u_\infty$ does not change sign. It then follows by Lemma~\ref{sec:nodal-solut-morse}(ii) that 
$u_\infty = w$, which contradicts (\ref{eq:claim-convergence-w-contradiction}). We thus have proved (\ref{eq:claim-convergence-w}), and in the same way it follows that 
\begin{equation}
  \label{eq:claim-convergence-w-minus}
\eta^{t}(u - \eps \phi_1) \to -w \qquad \text{in $C^1_0(\Omega)$ as $t \to \infty$.}  
\end{equation}
Combining (\ref{eq:claim-convergence-w}) and (\ref{eq:claim-convergence-w-minus}), we may thus define a continuous path $\gamma:[-1,1] \to H^1_0(\Omega)$ by setting
$$
\gamma(t)= \left \{
  \begin{aligned}
 &\gamma_0(t),&&\qquad t \in [-\eps,\eps],\\
 &\eta^{\tau_-(t)}(u + \eps \phi_1),&&\qquad t \in (\eps,1),\\ 
 &w,&& \qquad  t = 1,\\
 &\eta^{\tau_+(t)}(u - \eps \phi_1),&&\qquad t \in (-1,\eps),\\ 
 &-w,&&\qquad  t = -1,
  \end{aligned}
\right.
\qquad\qquad\quad  \text{where}\quad \tau_\pm(t)= \frac{1-\eps}{1\pm t}-1.
$$
By construction, $\gamma$ has the asserted properties.
\end{proof}

In combination with a highly useful result by Aftalion and Pacella \cite{AftalionPacella}, Theorems~\ref{prop:boundmorseindex} and \ref{sec:lens-mp} allow us to derive more information in the case where $\Omega$ is a bounded radial domain, i.e., if $\Omega$ is a ball or an annulus centered at zero. We need to recall the following definition. A function $u$ defined on a
radial domain is said to be {\em foliated Schwarz symmetric with respect to some unit vector $e\in \R^N$,
 $|e|=1$}, if $u(x)$ only depends on $r=|x|$ and $\theta
 =\arccos \bigl(\frac x{|x|}\cdot e\bigr)$, and $u$ is nonincreasing
 in $\theta$.

\begin{theorem}
\label{sec:ball-case}
If $\Omega$ is a bounded radial domain and $f$ satisfies (A1)-(A2)-(A3')-(A4), then we have the following. 
\begin{enumerate}
\item[(i)] $c_{nod}= c_{mp}$.
\item[(ii)] Every nodal solution of \eqref{eq:mainproblem} satisfies $m(u) \ge 1$.
\item[(iii)] Every nodal solution $u$ with $I(u)= c_{nod}$ is of mountain pass type
in the sense of Theorem~\ref{sec:lens-mp}. Moreover, $u$ is nonradial and foliated Schwarz symmetric with respect to some unit vector $e \in \R^N$ and  $u$ is strictly decreasing in the polar angle 
$\theta =\arccos \bigl(\frac x{|x|}\cdot e\bigr)$.
\end{enumerate}
\end{theorem}

\begin{proof}
We start by proving (ii). Suppose by contradiction that there exists a nodal solution $u$ of \eqref{eq:mainproblem} with $m(u)=0$. We consider cylinder coordinates in $\Omega$, 
replacing $x_1$, $x_2$ by polar coordinates
$r,\vartheta$ with $x_1= r \cos \vartheta$,
$x_2= r \sin \vartheta$, and leaving \mbox{$x':=(x_3,\dots,x_N)$} unchanged.
In these coordinates, we then have that
$$
\Delta u = \frac{\partial^2 u}{\partial r^2}+\frac{1}{r}
\frac{\partial u}{\partial r}+\frac{1}{r^2}\frac{\partial^2 u}{\partial \vartheta^2}
+\sum_{i=3}^{N}\frac{\partial^2 u}{\partial x_i^2}.
$$
We now consider the angular derivative $u_\vartheta = \frac{\partial u}{\partial \vartheta}: \overline \Omega \to \R$, and we claim that 
\begin{equation}
\label{angle-deriv-zero}
u_\vartheta \equiv 0.   
\end{equation}
Differentiating the equation $-\Delta u= f(u)$ with respect to $\vartheta$, we find that $u_\vartheta$ satisfies
\begin{equation}
  \label{eq:linearized-angle}
\left\{
  \begin{aligned}
-\Delta u_\vartheta&= f'(u)u_\vartheta &&\qquad \text{in $\Omega$,}\\
u_\vartheta & = 0 &&\qquad \text{on $\partial \Omega$.}    
  \end{aligned}
\right.
\end{equation}
Suppose by contradiction that $u_\vartheta \not \equiv 0$. Then $u_\vartheta$ is a Dirichlet eigenfunction of $-\Delta- f'(u)$ corresponding to the eigenvalue $0$. Moreover, since $u$ is $2\pi$-periodic in $\vartheta$, $u_\theta$ changes sign. This implies that the first Dirichlet eigenvalue of $-\Delta - f'(u)$ is negative, which contradicts our assumption that $m(u)=0$. 

Hence \eqref{angle-deriv-zero} is satisfied, and it also holds true for an arbitrary rotation $\tilde u$ of the function $u$ since $\tilde u$ satisfies the same assumptions as $u$. Consequently, $u$ is a radial sign changing solution of \eqref{eq:mainproblem}. However, then \cite[Theorem 1.1]{AftalionPacella} implies that $m(u) \ge N+1>0$, which is a contradiction. We thus conclude that $m(u) \ge 1$ for every nodal solution $u$ of \eqref{eq:mainproblem}, as claimed in (ii).\\
Next, let $u \in C^1_0(\Omega)$ be a nodal solution with $I(u)= c_{nod}$ (such a solution exists by Theorem~\ref{thm:lensachieved}). By (ii) and Theorem~\ref{prop:boundmorseindex}, we have $m(u)=1$, so (i) follows by Theorem~\ref{sec:lens-mp}. It also follows from Theorem~\ref{sec:lens-mp} that $u$ is of mountain pass type. Moreover, by \cite[Theorem 1.1]{AftalionPacella}, $u$ is nonradial. 

Next we prove that $u$ is foliated Schwarz symmetric. 
Let $x_0 \in \Omega \setminus \{0\}$ be chosen such that $u(x_0)=\max \{u(x)\::\: |x|=|x_0|\}$. We put $e:=\frac{x_0}{|x_0|}$, and we consider the family $\cH_e$ of all open halfspaces $H$ in
$\R^N$ with $e \in H$ and $0 \in \partial H$. For $H \in \cH_e$ we let $\sigma_H: \R^N \to \R^N$ denote the reflection 
with respect to the hyperplane $\partial H$. We claim the following:
\begin{equation}
  \label{eq:7}
\text{For every $H \in \cH_e$, we have $u \ge u \circ \sigma_H$ on $H \cap \Omega$.} 
\end{equation}
To prove this, we fix $H \in \cH_e$ and recall that the  {\em polarization} of $u$ with respect to $H$ is defined by
$$
u_H(x)=
\begin{cases}
 &\max \{u(x),u(\sigma_H(x))\},\qquad x\in \Omega \cap H\\
 &\min \{u(x),u(\sigma_H(x))\},
\qquad x\in \Omega \setminus H.
\end{cases}
$$
It is well known (see e.g. \cite{BartschWethWillem}) that
\begin{equation}
  \label{eq:equimeasurable}
\int_{\Omega} |\nabla u_H|^2\:dx = \int_{\Omega} |\nabla
u|^2\:dx \qquad \text{and}\qquad \int_{\Omega} F(u_H)\:dx = \int_{\Omega} F(u)\:dx. \end{equation}
By Theorem~\ref{sec:ball-case}(iii), $u$ is of mountain pass type, so there exists a path 
$$
\gamma: [-1,1] \to [-w,w] \subset C^1_0(\Omega) \qquad \text{with}\quad \gamma(\pm 1) = \pm w,\quad \gamma(0)= u
$$
and such that $I \circ \gamma$ is maximized precisely at $0$ with $I(\gamma(0))=I(u)=c_{nod}$. We then consider the {\em polarized path} 
$$
\gamma_*: [-1,1] \to H^1_0(\Omega) \quad \text{defined by}\quad  \gamma_*(t)=[\gamma(t)]_H.
$$
Since $\pm w$ are radial functions, $\gamma_*$ inherits the property that 
$\gamma_*(\pm 1) = \pm w$. Moreover, by \eqref{eq:equimeasurable}, $I \circ \gamma_*$ is also maximized precisely at $0$ with $I(\gamma_*(0))=c_{nod}$. From this we now deduce that 
\begin{equation}
  \label{eq:critical-point-path}
\text{$u_H= \gamma_*(0)$ is a solution of (\ref{eq:mainproblem}).}
\end{equation}
Indeed, suppose by contradiction that this is not the case, and consider again the flow defined in (\ref{eq:eta_def}). We recall that, by Lemma~\ref{sec:nodal-solut-morse}, $\eta^t$ is well defined as a map $[-w,w] \to [-w,w]$ for every $t>0$.
For fixed $\eps>0$, we then consider the path 
$$
\gamma_\eps: [-1,1] \to H^1_0(\Omega) \quad \text{defined by}\quad  \gamma_\eps(s)=\eta^\eps(\gamma_*(s)),
$$
which then also satisfies $\gamma_\eps(\pm 1) = \pm w.$ By definition of $c_{mp}$ and Theorem~\ref{sec:ball-case}(i), we have that 
$$
\max_{s \in [-1,1]} I(\gamma_\eps(s)) \ge c_{mp}=c_{nod}.
$$
Let $s_0 \in [-1,1]$ be chosen such that $I(\gamma_\eps(s_0)) \ge c_{nod}$, which implies that 
\begin{equation}
  \label{eq:ineq-equ-nod}
I(\gamma_*(s_0)) \ge I(\gamma_\eps(s_0)) \ge c_{nod}.
\end{equation}
Since  $I \circ \gamma_*$ is maximized precisely at $0$ with 
$I \circ \gamma_*(0)=c_{nod}$, it then follows that $s_0=0$, and equality holds in \eqref{eq:ineq-equ-nod}. From this we deduce -- by definition of $\gamma_*$ and $\gamma_\eps$ -- that $I(\eta^\eps(u_H))= I(u_H)$. Consequently, by Lemma~\ref{sec:nodal-solut-morse}(iii), $u_H$ is a solution of (\ref{eq:mainproblem}), as claimed in \eqref{eq:critical-point-path}.

We thus infer that both $u$ and $u_H$ are
solutions of \eqref{eq:mainproblem}. Therefore $z:=u_H-u$ is a
nonnegative function in $\Omega \cap H$ satisfying 
$$
-\Delta z= V(x)z \qquad \text{in
    $H \cap \Omega$}
$$
with $V \in L^\infty(\Omega)$ defined by $V(x)=\int_{0}^1 f'(u(x)+t z(x))\,dt$.   
The strong maximum principle then implies that either $z \equiv 0$ or $z>0$ in $H \cap \Omega$. The
second case is impossible since $x_0 \in H \cap \Omega$ and
$z(x_0)=u_H(x_0)-u(x_0)=0$ by the choice of $x_0$. We therefore conclude that
$z \equiv 0$, hence $u=u_H$ and \eqref{eq:7} holds.\\
By continuity, \eqref{eq:7} implies that $u$ is symmetric with respect to every
hyperplane containing $e$, so it is axially symmetric with respect
to the axis $e\R$. Thus $u(x)$ depends only on $r=|x|$ and $\theta
 =\arccos \left(\frac x{|x|}\cdot e\right)$. Moreover, it also
 follows from \eqref{eq:7} that $u$ is nonincreasing
 in the polar angle $\theta$. We thus conclude that $u$ is foliated Schwarz symmetric.

Finally, to show that $u$ is strictly decreasing in the polar angle, we pass to cylinder coordinates again. For this we assume, without loss of generality, that $u$ is foliated Schwarz symmetric with respect to $e=e_1$, the first coordinate vector. Replacing $x_1$, $x_2$ by polar coordinates
$r,\vartheta$ with $x_1= r \cos \vartheta$,
$x_2= r \sin \vartheta$, we then deduce from the foliated Schwarz symmetry that the angular derivative $u_\vartheta = \frac{\partial u}{\partial \vartheta}$ satisfies $u_\vartheta \le 0$ in the half domain $H_\Omega:= \{x \in \Omega\::\: x_2>0\}$ and $u_\vartheta \ge 0$ in $-H_\Omega$. Moreover, $u_\vartheta$ solves the Dirichlet problem \eqref{eq:linearized-angle}. By the strong maximum principle, it follows that either $u_\vartheta \equiv 0$ in $H_\Omega$ or $u_\vartheta < 0$ in $H_\Omega$. If $u_\vartheta \equiv 0$ in $H_\Omega$, it follows from the axial symmetry of $u$ with respect to $e_1$ that $u$ is a radial function, which contradicts what we have already proved. Hence $u_\vartheta < 0$ in $H_\Omega$, and again it follows from the axial symmetry that $u$ 
is strictly decreasing in the polar angle 
$\theta =\arccos \bigl(\frac x{|x|}\cdot e_1\bigr)$.

The proof is finished.
\end{proof}

\section{The Allen-Cahn equation} \label{sec:allencahn}
In this section we provide additional information in the particular case of the Allen-Cahn equation:
\begin{equation}\label{eq:AC}
\begin{cases}
-\Delta u=\lambda(u-|u|^{p-1}u)\\
 u\in H^1_0(\Omega)
 \end{cases}\qquad (p>1).
\end{equation}

From the previous sections, together with some additional standard computations, we have the following information about positive and sign-changing bounded solutions (see for example \cite{Berestycki1981, Berger1979}):

\begin{enumerate}[(a)]
\item any bounded solution $u$ of problem \eqref{eq:AC} satisfies $\|u\|_\infty\leq 1$;
\item there are no nontrivial bounded solutions for $0\leq \lambda\leq \lambda_1(\Omega)$;
\item given $\lambda>\lambda_1(\Omega)$, there exists a \emph{unique} bounded positive solution, which we denote by $w_\lambda$. 
\item for $ \lambda\leq \lambda_2(\Omega)$, there are no bounded sign-changing solutions.
\item for $\lambda>\lambda_2(\Omega)$ there exists at least one pair $(u,-u)$ of bounded sign-changing solutions.
\end{enumerate}

For each $\lambda>\lambda_2(\Omega)$, denote by $c_{nod}^\lambda$ and $c_{mp}^\lambda$ respectively the least energy nodal level \eqref{eq:cnod} and the mountain pass level \eqref{eq:cmp} associated to problem \eqref{eq:AC}. Both levels are achieved; if $\Omega=B_1$, then they coincide (recall Theorem \ref{sec:ball-case}), while there is a dumbbell-type domain domain where $c_{nod}^\lambda<c_{mp}^\lambda$ and there are at least two pairs $(u,-u)$ of sign-changing solutions (Theorem \ref{thm:cmp_neqc_nod})).

Here we will provide further information for $\lambda\sim \lambda_2(\Omega)$: in particular, if $|\lambda-\lambda_2(\Omega)|$ is sufficiently small (by a quantity depending on the domain), then $c_{nod}^\lambda=c_{mp}^\lambda$. After that we focus on the structure of the set of nodal solutions: in the next subsection we treat the particular case when $\Omega$ is either a ball or an annulus in $\R^N$, while the last subsection deals with the case of $\Omega$ being a square in dimension two.

Let us start by a general standard lemma.

\begin{lemma}\label{lemma:convergence} Let $\{u_\lambda\}\subset H^1_0(\Omega)\cap L^\infty(\Omega)$ be a family of sign-changing solutions to \eqref{eq:AC} for $\lambda\sim \lambda_2(\Omega)$. Then $u_\lambda \to 0$ in $C^{1,\alpha}(\overline{\Omega})$ as $\lambda\to \lambda_2(\Omega)$.
\end{lemma}
\begin{proof}
Since $-1\leq u_\lambda\leq 1$, we immediately have uniform bounds in $L^\infty(\Omega)$--norm for $\Delta u_\lambda$, and so $u_\lambda\to u$ in $C^{1,\alpha}(\overline \Omega)$ (up to a subsequence), with $-\Delta u=\lambda_2(u-|u|^{p-1}u)$. Thus $u$ is signed, and reasoning as in the final part of the proof of Theorem \ref{thm:lensachieved}, we must have $u\equiv 0$.
\end{proof}

\begin{theorem}\label{thm:Morselambdasmall}
There exists $\eps=\eps(\Omega)>0$ such that for $\lambda\in (\lambda_2(\Omega),\lambda_2(\Omega)+\eps)$ and $u$ any sign-changing solution $u$ satisfies 
\[
m(u) \geq 1.
\]
In particular if $u$ achieves $c_{nod}^\lambda$ then $m(u)=1$ and  
\[
c_{nod}^\lambda=c^\lambda_{mp}\qquad \text{ for } \lambda\in (\lambda_2(\Omega),\lambda_2(\Omega)+\eps).
\]
\end{theorem}
\begin{proof}
Take a family of sign-changing solutions $\{u_\lambda\}$ for $\lambda\sim \lambda_2(\Omega)$. From the previous lemma, we know that $u_\lambda\to 0$ in $C^{1,\alpha}(\overline \Omega)$ as $\lambda\to \lambda_2(\Omega)^+$.  Define
\[
Q_{u_\lambda}(v):=\int_\Omega |\nabla v|^2-\lambda \int_\Omega v^2 + p \lambda \int_\Omega |u_\lambda|^{p-1}v^2.
\]
and let $\varphi_1$ be the positive $L^2$-normalized first eigenfunction of $(\Delta,H^1_0(\Omega))$. Then
\[
Q_{u_\lambda}(\varphi_1)=\lambda_1(\Omega)-\lambda + p\lambda \int_\Omega \vphi_1^2 |u_\lambda|^{p-1} \to \lambda_1(\Omega)-\lambda_2(\Omega)<0.
\]
Therefore $m(u_\lambda)\geq 1$ for $\lambda\sim \lambda_2(\Omega)$. This combined with Theorem \ref{prop:boundmorseindex} yields that $m(u_\lambda)=1$ if $u_\lambda$ is a least energy nodal solution. From Theorem \ref{sec:lens-mp} we conclude that $u_\lambda$ is of mountain-pass type, and $c_{nod}^\lambda=c_{mp}^\lambda$.
\end{proof}

\begin{remark}
Observe that there is no contradiction between Theorem \ref{thm:Morselambdasmall} and Theorem \ref{thm:cmp_neqc_nod}. While the latter applied to \ref{eq:AC} states that, given $\lambda>0$, there exists a domain $\Omega$ such that $c_{nod}^\lambda<c_{mp}^\lambda$, the former theorem states that, given $\Omega$, $c_{nod}^\lambda=c_{mp}^\lambda$ for $\lambda\sim \lambda_2(\Omega)$.
\end{remark}

\begin{remark}
We have proved that, for $\lambda\sim \lambda_2(\Omega)$,  $\vphi_1$ is a direction of negativity for the associated functional. We have another direction of negativity, which is $w_\lambda$ (the unique positive solution):
\begin{align*}
Q_{u_\lambda}(w_\lambda)&=\int_\Omega |\nabla w_\lambda|^2-\lambda\int_\Omega w_\lambda^2 +p\lambda \int_\Omega w_\lambda^2 |u_\lambda|^{p-1}\\
			&\to \int_\Omega |\nabla w_{\lambda_2}|^2-\lambda_2 \int_\Omega w_{\lambda_2}^2=-\lambda_2 \int_\Omega w_{\lambda_2}^p<0
			\end{align*}
			as $\lambda\to \lambda_2(\Omega)$.
\end{remark}

\subsection{The case of a bounded radial domain in any dimension, $\lambda\sim \lambda_2(\Omega)$} 
In this section, $\Omega$ is either a ball or an annulus centered at the origin in $\R^N$, $N\geq 2$. 

Let $E_2$ denote the eigenspace associated to the second eigenvalue in $\Omega$,  i.e.,
\[
E_2=\{u\in H^1_0(\Omega):\ -\Delta u=\lambda_2 u\text{ in } \Omega\},
\]
and let $P_2:H^1_0(\Omega)\to E_2$ denote the projection map.
We start by characterizing the elements of $E_2$.

Given a direction $e\in \mathbb{S}^{N-1}$, we write
\[
\Omega=\Omega_e^+\cup \Omega_e^- \cup \Gamma_e,
\]
with
\[
\Omega^\pm_e=\{x\in \Omega:\ \pm x\cdot e>0\},\qquad \Gamma_e=\{x\in \Omega:\ x\cdot e=0\}.
\]
\begin{lemma}\label{lemma=E2}
Let $\Omega$ be either a ball or an annulus centered at the origin in $\R^N$, $N\geq 2$. Let $u\not \equiv 0$ be an element of $E_2$. Then $u$ is odd symmetric with respect to a half space $\{x\cdot e=0\}$ for some $e\in \mathbb{S}^{N-1}$, axially symmetric with respect to $\R e$, and moreover $\{u=0\}=\Gamma_e$. In particular, $\lambda_2(\Omega)=\lambda_1(\Omega_e^+)=\lambda_1(\Omega_e^-)$.
\end{lemma}
\begin{proof}
This result seems to be well known (at least in the case of the ball), but since we couldn't find an exact reference we give here a proof. 

Let $u\in E_2\backslash \{0\}$. We claim there exists $e\in \mathbb{S}^{N-1}$ such that $u$ is odd with respect to the hyperplane $\{x\cdot e=0\}$, axially symmetric with respect to $\R e$. If $\Omega$ is a ball and $N=2$, this is a direct consequence of \cite[Theorem 1.2]{Damascelli}. In the general case, by Theorem 5.1 in \cite{BartschWethWillem}, we know that $u$ is foliated Schwarz symmetric with respect to some $e\in \mathbb{S}^{N-1}$. In particular, $u$ is invariant under rotations around $\R e$. Then, by \cite[Proposition 2.1]{GrumiauTroestler}, $w$ is odd in the direction $e$.  

As a consequence of the last paragraph, $\Gamma_e \subseteq \{w=0\}$. So $-\Delta w=\lambda_2w$ in the half-domain $\Omega_e^+$, and $w=0$ on $\partial \Omega_e^+$, which means that $\lambda_2(\Omega)=\lambda_k(\Omega_e^+)$ for some $k\in \N$. On the other hand, by taking the odd extension of the first eigenfunction on $\Omega_e^+$ to the whole $\Omega$, it follows that $\lambda_2(\Omega)\leq \lambda_1(\Omega_e^+)$. Therefore  $\lambda_2(\Omega)=\lambda_1(\Omega_e^+)$ and  $w|_{\Omega_e^+}$ is a first eigenfunction in $\Omega_e^+$, and either $w>0$ or $w<0$ in $\Omega_e^+$. In conclusion, $\Gamma_e= \{w=0\}$.
\end{proof}

In $\Omega_e^+$, for $\lambda>\lambda_1(\Omega_e^+)=\lambda_2(\Omega)$, take the unique positive solution $w^+_e(\lambda)$ of $-\Delta u=\lambda(u-u^3)$, and, similarly, define $w^-_e(\lambda)$ on $\Omega_e^-$.
Observe that, by considering
\[
w_e(\lambda):=\begin{cases}
w^+_e & \text{ in } \Omega_e^+\\
-w^-_e & \text{ in } \Omega_e^-
\end{cases}
\]
we obtain a sign-changing solution of \eqref{eq:AC} in the whole $\Omega$, which converges in $C^{1,\alpha}(\overline \Omega)$--norm to 0 as $\lambda\to \lambda_2(\Omega)$. Given $e_1,e_2\in \mathbb{S}^{N-1}$ with $e_1\neq e_2$, $w_{e_1}(\lambda)$ is obtained from $w_{e_2}(\lambda)$ after a rotation around the origin. Thus, we have a  manifold of solutions bifurcating from $(\lambda_2,0)$:
\[
S=\{(\lambda,w_e(\lambda)):\ \lambda>\lambda_2(\Omega), e \in \partial B_1(0)\}.
\]
The aim of this section is to prove that, for $\lambda$ close to $\lambda_2(\Omega)$, $S$ describes all possible sign-changing solutions of \eqref{eq:AC}.

By combining Lemma \ref{lemma=E2} and \cite{CrandallRabinowitz} we see that, close to $\lambda_2(\Omega)$, when restricted to the space $\{u\in H^1_0(\Omega):\ u \text{ odd with respect to } e\}$,  all bounded sign-changing solutions belong to the $C^1$-curve $\lambda\mapsto w_e(\lambda)$. The following result completes the picture.
 
\begin{theorem}\label{thm:symmetry}
Let $\Omega$ be a ball or an annulus centered at the origin in $\R^N$ ($N\geq 2$). There exists $\eps=\eps(\Omega)$ such that, if  $\lambda\in ( \lambda_2(\Omega),\lambda_2(\Omega)+\eps)$ and $u$ a sign-changing solution of \eqref{eq:AC}, then $u$ and $P_{2}(u)$ have the same symmetries, which means that $u$ is odd symmetric with respect to a half space $\{x\cdot e=0\}$ for some $e\in \mathbb{S}^{N-1}$, axially symmetric with respect to $\R e$, and $\{u=0\}=\{x\cdot e=0\}$.

In particular,  for $\lambda \in (\lambda_2(\Omega),\lambda_2(\Omega)+\eps)$, $c^\lambda_{mp}=c^\lambda_{nod}$ and
\[
u \text{ is a bounded sign-changing solution of } \eqref{eq:AC} \iff (\lambda,u)\in S.
\]  
\end{theorem}

We remark that, for more general equations than \eqref{eq:AC}, the paper by Miyamoto \cite{Miyamoto} provides bifurcation results for the case when $\Omega$ is a ball or an annulus in \emph{dimension 2}. In particular, Theorem \ref{thm:symmetry} for $B_1(0)\subset \R^2$  is contained in \cite[Theorem 3.5]{Miyamoto}. The proof of Theorem \ref{thm:symmetry} does not rely on bifurcation results; instead we adapt  ideas from  \cite{Bonheureetal} to our context. We divide the proof in several lemmas.

 \begin{lemma}[{\cite[Lemma 3.1]{Bonheureetal}}]\label{lemma:Bonheureetal} Let $N\geq 2$. There exists $\delta>0$ such that if $\|a(x)-\lambda_2(\Omega)\|_{L^{N/2}}<\delta$ and $u\in H^1_0(\Omega)$ solves $-\Delta u=a(x)u$ in $\Omega$, then either $u\equiv 0$ or $P_2u\not\equiv 0$.
\end{lemma}

Next observe that, given $u$ a solution of $-\Delta u=\lambda(u-|u|^{p-1}u)$, then $\tilde u=u/\|u\|_{\infty}$ solves 
\begin{equation}\label{eq:equationrescailing}
-\Delta \tilde u=\lambda\tilde u-\lambda \|u\|_{\infty}^{p-1}|\tilde u|^{p-1}\tilde u.
\end{equation} Denote by $B_r$ the $H^1_0$--ball centered at 0 of radius $r$.

 \begin{lemma}\label{lemma:alternative}
 There exists $M, \eps>0$ such that, if $\lambda\in (\lambda_2(\Omega),\lambda_2(\Omega)+\eps)$ and $u\in H^1_0(\Omega)$ is a sign-changing solution of $-\Delta u=\lambda(u-|u|^{p-1}u)$, then
 \[
 \frac{u}{\|u\|_{\infty}}\in B_M \quad \text{ and } \quad P_2\left(\frac{u}{\|u\|_{\infty}}\right)\not\in B_{1/M}.
 \]
 \end{lemma}
 \begin{proof}
 Suppose, in view of a contradiction, the existence of $\lambda_n \searrow \lambda_2(\Omega)$ and $u_n\in H^1_0(\Omega)$ a sign-changing solution of $-\Delta u_n=\lambda_n(u_n-|u_n|^{p-1}u_n)$ with
 \begin{equation}\label{eq:contradictionargument}
 \frac{u_n}{\|u_n\|_{\infty}}\not\in B_n \quad \text{ or } \quad P_2 \left(\frac{u_n}{\|u_n\|_{\infty}}\right)\in B_{1/n}.
 \end{equation}
 Define $\tilde u_n:=u_n/\|u_n\|_{\infty}$, which solves \eqref{eq:equationrescailing} with a bounded right-hand side. Then by combining elliptic  estimates with Lemma \ref{lemma:convergence} we see that $\tilde u_n \to \tilde u$ in $C^{1,\alpha}(\overline \Omega)$, with $\|\tilde u\|_\infty=1$ and $-\Delta \tilde u=\lambda_2(\Omega) \tilde u$. This contradicts \eqref{eq:contradictionargument}.
 \end{proof}
 
 \begin{lemma}\label{lemma:aux1}
There exists $\eps>0$ such that, if $\lambda\in (\lambda_2(\Omega),\lambda_2(\Omega)+\eps)$, $u,v\in H^1_0(\Omega)$ are sign-changing solutions of \eqref{eq:AC} with $\|u\|_\infty=\|v\|_\infty$, then either $u\equiv v$ or $P_2 u\not\equiv P_2 v$.
 \end{lemma}
 
 \begin{proof}
Take $\lambda_n\to \lambda_2(\Omega)$, and let $u_n,v_n$ be solutions to \eqref{eq:AC}  such that $\|u_n\|_\infty=\|v_n\|_\infty$. Let $\tilde u_n:= u_n/\|u_n\|_\infty$ and $\tilde v_n:= v_n/\|v_n\|_\infty$, which solve \eqref{eq:equationrescailing}. By the previous lemma we have that, up to a subsequence, 
 \begin{equation}\label{eq:weakconvergencerescailing}
 \tilde u_n\rightharpoonup \alpha\neq 0,\qquad \tilde v_n\rightharpoonup \beta\neq 0 \text{ weakly in } H^1_0(\Omega).
 \end{equation}
 Thus, since $u_n,v_n \to 0$ by Lemma \ref{lemma:convergence}, we have $-\Delta \alpha=\lambda_2 \alpha$, $-\Delta \beta=\lambda_2 \beta$, and the convergence in \eqref{eq:weakconvergencerescailing} is actually strong. Using the fact that $u_n,v_n$ have the same $L^\infty$ norm, we write
\begin{align*}
-\Delta (\tilde u_n-\tilde v_n)&=\lambda_n(\tilde u_n-\tilde v_n)+ \lambda_n \|v_n\|_{\infty}^{p-1} |\tilde v_n|^{p-1}\tilde v_n-\lambda_n \|u_n\|_{\infty}^{p-1} |\tilde u_n|^{p-1}\tilde u_n\\
					&=\lambda_n (\tilde u_n-\tilde v_n +\|u_n\|_{\infty}^{p-1} (|\tilde v_n|^{p-1}\tilde v_n-|\tilde u_n|^{p-1}\tilde u_n))=a_n(x)(\tilde u_n-\tilde v_n),
\end{align*}
for
\[
a_n(x):=
\lambda_n\left(1- \|u_n\|_\infty^{p-1}\frac{|\tilde v_n|^{p-1}\tilde v_n-|\tilde u_n|^{p-1}\tilde u_n}{\tilde v_n-\tilde u_n}\right)
\]
Since $\|a_n(x)-\lambda_2(\Omega)\|_{\infty}\to 0$, then by Lemma \ref{lemma:Bonheureetal} either $\tilde u_n\equiv \tilde v_n$ or $P_2 \tilde u_n\not\equiv P_2 \tilde v_n$. The conclusion now follows by using once again the fact that $\|u_n\|_\infty=\|v_n\|_\infty$.
 \end{proof}
 
 \begin{lemma}\label{lemma:T}
Let $\eps>0$ be as in Lemma \ref{lemma:aux1} and let $\lambda\in(\lambda_2(\Omega),\lambda_2(\Omega)+\eps)$. Let $u\in H^1_0(\Omega)\cap L^\infty(\Omega)$ be a solution of \eqref{eq:AC}  such that
$P_2 u \neq 0$, and let $T:H^1_0(\Omega)\to H^1_0(\Omega)$ be a linear map such that
\begin{enumerate}
\item[(i)] $T(E_2)=E_2$, $T(E_2^\perp)=E_2^\perp$; $T(P_2 u)=P_2 u$;
\item[(ii)] $Tu$ solves \eqref{eq:AC};
\item[(iii)] $\|Tu\|_\infty=\|u\|_\infty$.
\end{enumerate}
Then, $T u=u$.
 \end{lemma}
 \begin{proof}
Consider the splitting $H^1_0(\Omega)=E_2\oplus E_2^\perp$ and observe that
\begin{align*}
P_2(Tu)+&P_{E_2^\perp} (T u )=T (u ) =T( P_2 u) +T u .
\end{align*}
In particular, by property (i), $P_{2}(T u)=T( P_2 u)=P_2 u$. Applying Lemma \ref{lemma:aux1} to $u,Tu$ (and using (ii)), we deduce that $Tu=u$.
 \end{proof}

 \medbreak
 
 \begin{proof}[Proof of Theorem \ref{thm:symmetry}]
 Take $\eps>0$ as in Lemma \ref{lemma:aux1}, and let $u$ be a sign changing solution of \eqref{eq:AC} for $\lambda\in (\lambda_2(\Omega),\lambda_2(\Omega)+\eps)$. Lemma \ref{lemma:alternative} implies that $\alpha:=P_2u \not\equiv 0$. By Lemma \ref{lemma=E2}, we know that $\alpha$ is odd with respect to a certain direction, being even in all the other orthogonal directions. Suppose, without loss of generality, that
 \begin{align*}
& \alpha(-x_1,x_2,\ldots, x_N) =-\alpha(x_1,x_2,\ldots, x_N),\\
 & \alpha(x_1,x_2,\ldots, -x_i,\ldots, x_N) =\alpha(x_1,x_2,\ldots, x_i,\ldots, x_N),\ \forall i=2,\ldots, N.
 \end{align*}
 Take the linear map
 \[
 T:H^1_0(\Omega)\to H^1_0(\Omega),\qquad Tv(x):=-v(-x_1,x_2,\ldots, x_N),
 \]
which satisfies the conditions of Lemma  \ref{lemma:T}. Thus $Tu=u$, that is, $u$ is odd with respect to the same direction of $\alpha$. The fact that $u$ is even with respect to all other orthogonal directions is analogous, working this time with $T_iv(x):=v(x_1,x_2,\ldots, -x_i,\ldots, x_N)$.
 \end{proof}

 \subsection{The case of a square, $\lambda\sim \lambda_2(\Omega)$} 
 
 In this part we investigate \eqref{eq:AC} in the particular case of $p=3$, namely
 \begin{equation}\label{eq:ACsquare}
-\Delta v=\lambda(v-v^3), \qquad v \in H^1_0(\Omega),
\end{equation}
 in the square $\Omega = (0, \pi) \times (0, \pi)$. Observe that in this domain we are not in the situation of the previous sections, where $\partial \Omega$ is supposed to be smooth. However, for $\lambda\sim \lambda_2(\Omega)$ we are able to completely characterize the solution set, and in particular to determine the shape of the least energy nodal solution.
 
 By means of the change of variables $u = \lambda^{1/2}v$, we infer that \eqref{eq:ACsquare} is equivalent to
  \begin{equation}\label{eq:realACsquare}
-\Delta u=\lambda u- u^3, \qquad u \in H^1_0(\Omega),
\end{equation}
and we keep the form \eqref{eq:realACsquare} in order to take advantage of the results in \cite{delPinoGarciaMelianMusso}. In view of our purposes it is essential to observe that the Morse indices of $u$ and $v$ are the same, for $u = \lambda^{1/2}v$, when $v$ solves \eqref{eq:realACsquare}.

\begin{lemma}\label{lemma:morseindex}
Let $v$ be a solution of \eqref{eq:realACsquare}. Then $u = \lambda^{1/2}v$ solves \eqref{eq:ACsquare}, and $u$ and $v$ have the same Morse index.
\end{lemma}
\begin{proof} Indeed, since $u = \lambda^{1/2}v$, the pair $(\phi, \mu)$ solves
$\Delta \phi +\lambda (1- 3 v^2) \phi + \mu \phi=0$ if, and only if, it solves $\Delta \phi + (\lambda- 3 u^2) \phi + \mu \phi=0.$
\end{proof}

From the results in \cite{delPinoGarciaMelianMusso}, for $\lambda> \lambda_2$ and $\lambda\sim \lambda_2$,  arising from bifurcation branches, \eqref{eq:realACsquare} has exactly eight nodal solutions, coming in four pairs $(u,-u)$. In order to be more precise, let us first we recall that all the second eigenfunctions $\phi_{\alpha}(x,y)$ of the Laplacian in $\Omega$, with $\int_{\Omega} \phi_{\alpha}^2 = \frac{\pi^2}{4}$, are
\begin{equation}\label{eigenfunction-square}
\phi_{\alpha}(x, y) = \cos(\alpha) \sin (x ) \sin(2y) + \sin(\alpha) \sin(2x) \sin(y).
\end{equation}
 We also introduce the eigenfunction
\begin{equation}\label{psi-eigenfunction-square}
\psi_{\alpha}(x, y) = \sin(\alpha) \sin (x ) \sin(2y) - \cos(\alpha) \sin(2x) \sin(y),
\end{equation}
orthogonal to $\phi_{\alpha}$. Next we recall the following qualitative result from \cite{delPinoGarciaMelianMusso}, see \cite[Theorem 1.1]{delPinoGarciaMelianMusso} and some remarks at \cite[p. 3501]{delPinoGarciaMelianMusso}, based on standard bifurcation analysis.

\begin{theorem}[{\cite[Theorem 1.1]{delPinoGarciaMelianMusso}}]\label{th:delpino}
There exists $\eps >0$ and a neighborhood $\mathcal{U}$ of $(\lambda_2,0)$ in $\mathbb{R} \times C(\overline{\Omega})$ such that the set of all solutions of \eqref{eq:realACsquare} in $\mathcal{U}$ can be described as the union of four $C^1$--curves in $\mathbb{R} \times C(\overline{\Omega})$,
\[
s \in (-\eps, \eps) \mapsto (\lambda_i(s), u_i(s)), \quad i=1,\ldots, 4,
\]
such that
\begin{equation}\label{square:asymptotics}
\begin{cases}
(\lambda_i(-s), u_i(-s)) = (\lambda_i(s), -u_i(s)),\\
\lambda_i(s) = \lambda_2+ \sigma_i s^2 + o(s^2),\\
u_i(s) = s \phi_{\alpha_i} + o(s),
\end{cases}
\end{equation}
where $\alpha_1=0$, $\alpha_2= \pi/4$, $\alpha_3=\pi/2$, $\alpha_4=3\pi/4$ and
\begin{equation}\label{sigmas}
\sigma_i = \frac{9}{16} \ \ \hbox{for} \ \ i=1,3, \qquad \sigma_i = \frac{21}{32} \ \ \hbox{for} \ \ i = 2, 4.
\end{equation}
\end{theorem}

Observe that these solutions are of two types. For $\alpha=0,\pi/2$, their nodal set is respectively $\{x=\pi/2\}$ and $\{y=\pi/2\}$; we call these solutions of type M. For $\alpha=\pi/4,3\pi/4$, their nodal set is respectively $\{x=y\}$ and $\{y=\pi-x\}$; we call these solutions of type D. Here we distinguish the least energy nodal solutions, by counting the Morse indices of these solutions.

\begin{theorem}\label{th:square}
For $\lambda> \lambda_2$ and $\lambda\sim \lambda_2$, it holds:
 \begin{enumerate}[a)]
 \item Solutions of type M have Morse index one.
 \item Solutions of type D have Morse index two.
 \item Least energy nodal solutions of \eqref{eq:realACsquare} are of type M.
 \end{enumerate}
\end{theorem}

\begin{remark} Surprisingly, by replacing $-u^3$ by $u^3$ in \eqref{eq:ACsquare}, it is proved in \cite[Theorem 1.1]{delPinoGarciaMelianMusso} that solutions of type M have \emph{lower} Morse index than solutions of type $D$, therefore the sign of the nonlinearity alters the structure of the least energy nodal solution.
\end{remark}

\begin{proof}[Proof of Theorem \ref{th:square}] Here, supported by Theorem \ref{th:delpino}, we compute the Morse indices of these solutions, by mimicking the arguments in \cite[Section 3]{delPinoGarciaMelianMusso}, and also their energies.

We must compute the number of negative eigenvalues $\mu$ of 
\[
\Delta \phi + \lambda\phi - 3 u^2 \phi + \mu \phi=0  \ \ \text{in} \ \ \Omega, \quad \phi = 0 \ \ \text{on} \ \ \partial \Omega.
\]
Setting $u = s \phi_{\alpha} + \psi_s$, $\lambda = \lambda_2 + \sigma s^2$  with $\psi_s = o(1)$ and
\begin{equation}\label{eq:sigmadef}
\displaystyle{\ \left.\sigma = \int \phi_{\alpha}^4\large \middle/\int \phi_{\alpha}^2\right.},
\end{equation}
we obtain
\begin{equation}\label{eq:eigenvaluecalculation}
\Delta \phi + \lambda_2 \phi + s^2\sigma \phi - 3 s^2(\phi_{\alpha}+ s^2 \psi_s)^2 \phi + \mu \phi = 0 \ \ \text{in} \ \ \Omega, \quad \phi = 0 \ \ \text{on} \ \ \partial \Omega,
\end{equation}
where $\phi = \phi_s$ and $\mu = \mu_s$ and we normalize $\int \phi_s^2 = \pi^2/4$. As $s$ goes to zero, $(\phi_s, \mu_s) \to (\phi, \mu)$ that solves
\[
\Delta \phi + (\lambda_2 + \mu) \phi = 0 \ \ \text{in} \ \ \Omega, \quad \phi = 0 \ \ \text{on} \ \ \partial \Omega,
\]
whose set of eigenvalue is $\{ \lambda_1 - \lambda_ 2, 0, \lambda_3, \ldots\}$, where $\lambda_i$ are the eigenvalues of $(\Delta, H^1_0(\Omega)$. Therefore, for $s \sim 0$, the first eigenvalue $\mu$ of \eqref{eq:eigenvaluecalculation} is negative, and since $\lambda_2$ has multiplicity two, there are two others eigenvalues (counting their multiplicity) that are close to zero, and we must investigate their sign, and all the others eigenvalues are positive. We denote these two eigenvalues by $\mu_{1,s}$ and $\mu_{2,s}$, and let $\phi_{1,s}$ and $\phi_{2,s}$ be the corresponding eigenfunctions, with $\int \phi_{1,s} \phi_{2,s}=0$ and $\int\phi_{i,s}^2 = \pi^2/4$ for $i=1,2$. Then $\phi_{i,s} \to A_i \phi_{\alpha} + B_i \psi_{\alpha}$, $\mu_{i,s} \to 0$ as $s \to 0$, for $i=1,2$, with
\begin{equation}\label{eq:conditionsAB}
A_i^2 + B_i^2 =1 \ \ \text{and} \ \ A_1A_2 + B_1 B_ 2 = 0.
\end{equation}
Multiplying \eqref{eq:eigenvaluecalculation} by $\phi_{\alpha}$ and integrating by parts we infer that
\[
s^2\sigma \int \phi_{i,s} \phi_{\alpha} - 3 s^2\int (\phi_{\alpha} + s^2 \phi_s)^2 \phi_{i,s} \phi_{\alpha} + \mu_{i,s} \int \phi_{i,s} \phi_{\alpha} = 0.
\]
Then, dividing this equation by $\sigma s^2 \int \phi_{\alpha}^2$ and taking $s \to 0$, we obtain
\begin{equation}\label{eq:forAi}
A_i\left(  -2 + \frac{1}{\sigma}\lim_{s \to 0} \frac{\mu_{i,s}}{s^2}  \right) = 0, \ \ \text{for} \ \ i =1, 2,
\end{equation}
since $\int \phi_{\alpha}^3 \psi_{\alpha} = 0$ and $\sigma \int \phi_{\alpha}^2  = \int \phi_{\alpha}^4$. Next, multiplying \eqref{eq:eigenvaluecalculation} by $\psi_{\alpha}$ and integrating by parts we infer that
\[
s^2\sigma \int \phi_{i,s} \psi_{\alpha} - 3 s^2\int (\phi_{\alpha} + s^2 \phi_s)^2 \phi_{i,s} \psi_{\alpha} + \mu_{i,s} \int \phi_{i,s} \psi_{\alpha} = 0.
\]
Then, dividing this equation by $\sigma s^2 \int \psi_{\alpha}^2=\sigma s^2 \int \phi_{\alpha}^2 = s^2 \int \phi_{\alpha}^4$ and taking $s \to 0$, since $\int \phi_{\alpha}^3 \psi_{\alpha} = 0$, we obtain
\begin{equation}\label{eq:forBi}
B_i\left(  1- 3\frac{\int \phi_{\alpha}^2\psi_{\alpha}^2}{\int\phi_{\alpha}^4} + \frac{1}{\sigma}\lim_{s \to 0} \frac{\mu_{i,s}}{s^2}  \right) = 0, \ \ \text{for} \ \ i =1, 2.
\end{equation}
On the other hand,
\begin{equation}\label{eq:valuesL2L4}
\int \phi_{\alpha}^2 = \frac{\pi^2}{4}, \quad \int\phi_{\alpha}^4 = \frac{3\pi^2}{256}(13 - \cos(4 \alpha)) \ \ \text{and} \ \ 3 \int\phi_{\alpha}^2 \psi_{\alpha}^2 = \frac{3\pi^2}{256}(13 + 3 \cos(4 \alpha)).
\end{equation}
Then, from \eqref{eq:forAi} and \eqref{eq:forBi} we infer that
\begin{equation}\label{eq:finalAiBimed}
\begin{cases}
\displaystyle A_i \left( -2 + \frac{16}{9} \lim_{s \to 0} \frac{\mu_{i,s}}{s^2} \right) = 0  \ \ \text{and} \ \ B_i\left( -\frac{1}{3} + \frac{1}{\sigma}\lim_{s \to 0} \frac{\mu_{i,s}}{s^2}  \right) = 0 \vspace{5pt}\\ \displaystyle \text{for $i= 1,2$, and $\alpha= 0$ or $\alpha=\frac{\pi}{2}$}
\end{cases}
\end{equation}
and
\begin{equation}\label{eq:finalAiBidia}
\begin{cases}
\displaystyle A_i \left( -2 + \frac{32}{21} \lim_{s \to 0} \frac{\mu_{i,s}}{s^2} \right) = 0  \ \ \text{and} \ \ B_i\left( \frac{2}{7} + \frac{32}{21}\lim_{s \to 0} \frac{\mu_{i,s}}{s^2}  \right) = 0 \vspace{5pt}\\ \displaystyle \text{for $i= 1,2$, and $\alpha= \frac{\pi}{4}$ or $\alpha=\frac{3\pi}{4}$}.
\end{cases}
\end{equation}
From \eqref{eq:finalAiBimed} we conclude that $\mu_{1,s}>0$ and $\mu_{2,s}>0$ for $s\sim 0$, and so that the Morse index of $u$ is one in case that $\alpha = 0$ or $\alpha = \frac{\pi}{2}$. On the other hand, from \eqref{eq:finalAiBidia} we obtain that the product $\mu_{1,s} \cdot \mu_{2,s}$ is negative for $s \sim 0$, and so that the Morse index of $u$ is two in case $\alpha = \frac{\pi}{4}$ or $\alpha = \frac{3 \pi}{4}$. This proves (a) and (b).

To conclude, we show which branches correspond to least energy nodal solutions. Observe that, whenever $v$ is a bounded solution of \eqref{eq:ACsquare}, its energy can be written as
\[
J(v)=\frac{1}{2}\int_\Omega (|\nabla v|^2 - \lambda v^2) + \frac{\lambda}{4}\int_\Omega v^4=-\frac{\lambda}{4}\int_\Omega v^4.
\]
Therefore, using \eqref{square:asymptotics}, for $\lambda\sim \lambda_2(\Omega)$ we see that
\begin{align*}
J(\lambda^{-1/2}u_i) \sim -\frac{1}{4\lambda} (\lambda-\lambda_2(\Omega))^2 \frac{1}{\sigma_i^2}\int_\Omega \phi_{\alpha_i}^4 =\begin{cases}
-\frac{\pi^2}{9} (\lambda-\lambda_2(\Omega))^2/\lambda  & \text{ if } i=1,3\vspace{5pt}\\
-\frac{2\pi^2}{21}(\lambda-\lambda_2(\Omega))^2/\lambda  & \text{ if } i=2,4
\end{cases},
\end{align*}
which proves (c).

\end{proof}

\subsection*{Acknowledgements}
Ederson Moreira dos Santos was partially supported by CNPq grant 307358/2015--1 and FAPESP grant 2016/50453--0. Hugo Tavares was supported by the Portuguese government through FCT -- Funda\c c\~ao para a Ci\^encia e a Tecnologia, I.P., under the projects UIDB/MAT/04459/2020, \linebreak UID/MAT/04561/2013UID/MAT/04561/2013 and PTDC/MAT-PUR/28686/2017. Hugo Tavares would also like to acknowledge the Faculty of Sciences of the University of Lisbon for granting a semestral sabbatical leave, during which part of this work was developed.

\end{document}